\documentclass[12pt,oneside,english]{amsart}
\usepackage[latin9]{inputenc}
\usepackage{geometry}
\usepackage{amsthm}
\usepackage{amstext}
\usepackage{amssymb}
\usepackage{esint}
\usepackage[all]{xy}

\makeatletter
\numberwithin{equation}{section}
\numberwithin{figure}{section}
\theoremstyle{plain}
\newtheorem{thm}{\protect\theoremname}[section]
  \theoremstyle{remark}
  \newtheorem{rem}[thm]{\protect\remarkname}
  \theoremstyle{plain}
  \newtheorem{conjecture}[thm]{\protect\conjecturename}
  \theoremstyle{plain}
  \newtheorem{prop}[thm]{\protect\propositionname}
  \theoremstyle{definition}
  \newtheorem{defn}[thm]{\protect\definitionname}
  \theoremstyle{plain}
  \newtheorem{lem}[thm]{\protect\lemmaname}
  \theoremstyle{definition}
  \newtheorem{example}[thm]{\protect\examplename}
  \theoremstyle{plain}
  \newtheorem{cor}[thm]{\protect\corollaryname}


\makeatother

\usepackage{babel}
  \providecommand{\conjecturename}{Conjecture}
  \providecommand{\corollaryname}{Corollary}
  \providecommand{\definitionname}{Definition}
  \providecommand{\examplename}{Example}
  \providecommand{\lemmaname}{Lemma}
  \providecommand{\propositionname}{Proposition}
  \providecommand{\remarkname}{Remark}
\providecommand{\theoremname}{Theorem}

\begin{document}

\title{Wilder McKay correspondences}

\author{Takehiko Yasuda}
\begin{abstract}
A conjectural generalization of the McKay correspondence in terms
of stringy invariants to arbitrary characteristic, including the wild
case, was recently formulated by the author in the case where the
given finite group linearly acts on an affine space. In cases of very
special groups and representations, the conjecture has been verified
and related stringy invariants have been explicitly computed. In this
paper, we try to generalize the conjecture and computations to more
complicated situations such as non-linear actions on possibly singular
spaces and non-permutation representations of non-abelian groups.
\end{abstract}

\address{Department of Mathematics, Graduate School of Science, Osaka University,
Toyonaka, Osaka 560-0043, Japan, tel:+81-6-6850-5326, fax:+81-6-6850-5327}

\email{takehikoyasuda@math.sci.osaka-u.ac.jp}

\thanks{This work was partially supported by Grants-in-Aid for Scientific
Research (22740020).}

\maketitle
\global\long\def\AA{\mathbb{A}}
\global\long\def\PP{\mathbb{P}}
\global\long\def\NN{\mathbb{N}}
\global\long\def\GG{\mathbb{G}}
\global\long\def\ZZ{\mathbb{Z}}
\global\long\def\QQ{\mathbb{Q}}
\global\long\def\CC{\mathbb{C}}
\global\long\def\FF{\mathbb{F}}
\global\long\def\LL{\mathbb{L}}
\global\long\def\RR{\mathbb{R}}
\global\long\def\MM{\mathbb{M}}
\global\long\def\SS{\mathbb{S}}

\global\long\def\bx{\mathbf{x}}
\global\long\def\bf{\mathbf{f}}
\global\long\def\ba{\mathbf{a}}
\global\long\def\bs{\mathbf{s}}
\global\long\def\bt{\mathbf{t}}
\global\long\def\bw{\mathbf{w}}
\global\long\def\bb{\mathbf{b}}
\global\long\def\bv{\mathbf{v}}
\global\long\def\bp{\mathbf{p}}
\global\long\def\bm{\mathbf{m}}
\global\long\def\bj{\mathbf{j}}
\global\long\def\bM{\mathbf{M}}

\global\long\def\cN{\mathcal{N}}
\global\long\def\cW{\mathcal{W}}
\global\long\def\cY{\mathcal{Y}}
\global\long\def\cM{\mathcal{M}}
\global\long\def\cF{\mathcal{F}}
\global\long\def\cX{\mathcal{X}}
\global\long\def\cE{\mathcal{E}}
\global\long\def\cJ{\mathcal{J}}
\global\long\def\cO{\mathcal{O}}
\global\long\def\cD{\mathcal{D}}
\global\long\def\cZ{\mathcal{Z}}
\global\long\def\cR{\mathcal{R}}
\global\long\def\cC{\mathcal{C}}

\global\long\def\fs{\mathfrak{s}}
\global\long\def\fp{\mathfrak{p}}
\global\long\def\fm{\mathfrak{m}}
\global\long\def\fX{\mathfrak{X}}
\global\long\def\fV{\mathfrak{V}}
\global\long\def\fx{\mathfrak{x}}
\global\long\def\fv{\mathfrak{v}}
\global\long\def\fY{\mathfrak{Y}}

\global\long\def\rv{\mathbf{\mathrm{v}}}
\global\long\def\rx{\mathrm{x}}
\global\long\def\rw{\mathrm{w}}
\global\long\def\ry{\mathrm{y}}
\global\long\def\rz{\mathrm{z}}
\global\long\def\bv{\mathbf{v}}
\global\long\def\bx{\mathbf{x}}
\global\long\def\bw{\mathbf{w}}
\global\long\def\sv{\mathsf{v}}
\global\long\def\sx{\mathsf{x}}
\global\long\def\sw{\mathsf{w}}

\global\long\def\Spec{\mathrm{Spec}\,}
\global\long\def\Hom{\mathrm{Hom}}

\global\long\def\Var{\mathrm{Var}}
\global\long\def\Gal{\mathrm{Gal}}
\global\long\def\Jac{\mathrm{Jac}}
\global\long\def\Ker{\mathrm{Ker}}
\global\long\def\Im{\mathrm{Im}}
\global\long\def\Aut{\mathrm{Aut}}
\global\long\def\st{\mathrm{st}}
\global\long\def\diag{\mathrm{diag}}
\global\long\def\characteristic{\mathrm{char}}
\global\long\def\tors{\mathrm{tors}}
\global\long\def\sing{\mathrm{sing}}
\global\long\def\red{\mathrm{red}}
\global\long\def\Ind{\mathrm{Ind}}
\global\long\def\nr{\mathrm{nr}}
\global\long\def\ord{\mathrm{ord}}
\global\long\def\pt{\mathrm{pt}}
\global\long\def\op{\mathrm{op}}
 \global\long\def\univ{\mathrm{univ}}
\global\long\def\length{\mathrm{length}}
\global\long\def\sm{\mathrm{sm}}
\global\long\def\top{\mathrm{top}}
\global\long\def\rank{\mathrm{rank}}
\global\long\def\Mot{\mathrm{Mot}}
\global\long\def\age{\mathrm{age}\,}
\global\long\def\et{\mathrm{et}}
\global\long\def\hom{\mathrm{hom}}
\global\long\def\tor{\mathrm{tor}}
\global\long\def\reg{\mathrm{reg}}

\global\long\def\Conj#1{\mathrm{Conj}(#1)}
\global\long\def\Mass#1{\mathrm{Mass}(#1)}
\global\long\def\Inn#1{\mathrm{Inn}(#1)}
\global\long\def\bConj#1{\mathbf{Conj}(#1)}
\global\long\def\Hilb{\mathrm{Hilb}}
\global\long\def\sep{\mathrm{sep}}
\global\long\def\GL#1#2{\mathrm{GL}_{#1}(#2)}
\global\long\def\codim{\mathrm{codim}}

\global\long\def\GEx{G\text{-}\mathrm{Ex}}
\global\long\def\GCov{G\text{-}\mathrm{Cov}}
\global\long\def\preuntwisting{\left\langle F\right\rangle }
\global\long\def\Tr{\mathrm{Tr}}
\global\long\def\Nr{\mathrm{Nr}}
\global\long\def\Eta{\mathrm{Eta}}
\global\long\def\Fie{\mathrm{Fie}}

\tableofcontents{}

\section{Introduction}

The McKay correspondence in terms of stringy invariants was first
studied by Batyrev and Dais \cite{MR1404917}, and Batyrev \cite{MR1677693}.
Denef and Loeser \cite{MR1905024} later took a more conceptual approach,
where the McKay correspondence directly follows from the theory of
motivic integration suitably generalized to a situation involving
finite group actions. 

These works were confined to characteristic zero. Except some works
for the tame case (the finite group has order coprime to the characteristic),
an attempt of generalization to arbitrary characteristics, including
the wild (non-tame) case, was only recently started in \cite{Yasuda:2012fk,Yasuda:2013fk}.
There it was formulated a conjectural generalization of results in
characteristic zero. Subsequently it turned out in \cite{Wood-Yasuda-I}
that the conjecture is closely related to the number theory, in particular,
the problem of counting local Galois representations. In these papers,
however, only linear actions on affine spaces have been discussed.
In characteristic zero, since every finite group action on a smooth
variety is locally linearizable, many studies can be reduced to the
linear case. This is no longer true in positive or mixed characteristic.
The conjecture has been verified in very special cases, by computing
stringy invariants explicitly. The aims of this paper are firstly
to generalize the conjecture to non-linear actions on a (possibly
singular) affine variety, and secondly to make it possible to compute
stringy invariants in more complicated examples.

We now recall the conjecture from \cite{Yasuda:2013fk}. Set the base
scheme to be $D=\Spec\cO_{D}$ with $\cO_{D}$ a complete discrete
valuation ring and suppose that its residue field, denoted by $k$,
is algebraically closed. 
\begin{rem}
Working over a discrete valuation ring rather than a field is natural
in our arguments. We can easily switch from a field to a discrete
valuation ring by the base change by $\Spec k[[t]]\to\Spec k$.
\end{rem}
We consider a linear action of a finite group $G$ on the affine $d$-space
$V=\AA_{D}^{d}$ over $D$ and the associated quotient scheme $X:=V/G$. 
\begin{conjecture}[The wild McKay correspondence conjecture\emph{ \cite{Yasuda:2013fk}}]
\label{conj: linear McKay intro}Let $o\in X(k)$ denote the image
of the origin and $M_{\st}(X)_{o}$ the stringy motif of $X$ at $o$.
Suppose that the quotient morphism $V\to X$ is \'{e}tale in codimension
one. Then 
\[
M_{\st}(X)_{o}=\int_{\GCov(D)}\LL^{\bw}\, d\tau.
\]
Here $\GCov(D)$ is the (conjectural) moduli space of $G$-covers
of $D$, $\bw$ is the \emph{weight function} on $\GCov(D)$ associated
to the $G$-representation $V$ and $\tau$ is the tautological motivic
measure on $\GCov(D)$. 
\end{conjecture}
In \cite{Yasuda:2012fk}, Conjecture \ref{conj: linear McKay intro}
was verified, when $\cO_{D}=k[[t]]$ with $k$ of characteristic $p>0$
and $G$ is the cyclic group of order $p$. In \cite{Wood-Yasuda-I},
a variant conjecture was verified when the symmetric group $S_{n}$
acts on $\AA_{D}^{2n}$ by two copies of the standard representation.
In the same paper, a generalization to the case where $k$ is only
perfect was formulated by modifying the function $\bw$. 

Roughly the conjecture was derived as follows: we first express $M_{\st}(X)_{o}$
as a motivic integral over the space of \emph{arcs }of $X$, that
is, $D$-morphisms $D\to X$. We then transform the motivic integral
to a motivic integral over the space of $G$\emph{-arcs} of $V$,
that is, $G$-equivariant $D$-morphisms $E\to V$ for $G$-covers
$E\to D$. We can see that the contribution of each $G$-cover $E\to D$
to $M_{\st}(X)_{o}$ is $\LL^{\bw(E)}$, and hence the conjecture.
At this last point, we used the technique of \emph{untwisting, }which
enables us to reduce the study of $G$-arcs to the one of ordinary
arcs. A prototype of untwisting was introduced by Denef and Loeser
\cite{MR1905024}. In \cite{Yasuda:2013fk}, the author developed
it so that we can use it even in the wild case. In this paper, we
will refine the technique slightly more. For each $G$-cover $E$
of $D$ with a connected component $F$, we can construct another
affine space $V^{|F|}\cong\AA_{D}^{d}$ and a morphism $V^{|F|}\to X$
such that there is a correspondence between $G$-arcs of $V$ and
ordinary arcs of $V^{|F|}$. Through the correspondence, we can represent
the contribution of $E$ to $M_{\st}(X)_{0}$ as a motivic integral
over the ordinary arcs $D\to V^{|F|}$. 

Our strategy of generalization to the non-linear case is quite simple:
given an affine variety $\sv$ with a $G$-action, we equivariantly
embed $\sv$ into an affine space $V$ with a linear $G$-action.
For each $G$-cover $E\to D$ with a connected component $F$, we
take the subvariety $\sv^{|F|}\subset V^{|F|}$ corresponding to $\sv\subset V$,
which plays the same role as $V^{|F|}$ in the linear case.

We also need an idea from the minimal model program, that is, working
with varieties endowed with divisors rather than varieties themselves.
Encapsulating one more information, we will introduce the notion of
\emph{centered log structures }or \emph{centered log $D$-varieties},
which are just triples $\fX=(X,\Delta,W)$ of a normal $D$-variety
$X$, a $\QQ$-divisor $\Delta$ and a closed subset $W$ of $X\otimes_{\cO_{D}}k$
with $K_{X/D}+\Delta$ $\QQ$-Cartier. It is straightforward to generalize
the stringy motif to centered log $D$-varieties. We write it as $M_{\st}(\fX)$.
For instance, the stringy motif $M_{\st}(X)_{o}$ mentioned above
is the same as $M_{\st}((X,0,\{o\}))$. 

Returning to the equivariant immersion $\sv\hookrightarrow V$, if
$\sv$ is given a centered log structure $\fv=(\sv,\delta,\sw)$,
then there exist unique centered log structures $\fx$ on $\sx:=\sv/G$
and $\fv^{|F|,\nu}$ on the normalization $\sv^{|F|,\nu}$ of $\sv^{|F|}$
so that all the morphisms connecting them are crepant (see Section
\ref{sub:Crepant-morphisms} for details). If $H\subset G$ is the
stabilizer of the component $F\subset E$, then the centralizer $C_{G}(H)$
of $H$ acts on $\sv^{|F|,\nu}$ and on its arc space $J_{\infty}\sv^{|F|,\nu}$.
We will define $M_{\st,C_{G}(H)}(\fv^{|F|,\nu})$ in the same way
as defining the ordinary stringy motif except that we will use the
quotient space $(J_{\infty}\sv^{|F|,\nu})/C_{G}(H)$ rather than the
arc space $J_{\infty}\sv^{|F|,\nu}$ itself. We formulate the following
conjecture which generalizes Conjecture \ref{conj: linear McKay intro}. 
\begin{conjecture}[Conjecture \ref{conj: non-linear McKay-1}]
\label{conj: in intro}We have
\[
M_{\st}(\fx)=\int_{\GCov(D)}M_{\st,C_{G}(H)}(\fv^{|F|,\nu})\, d\tau.
\]

\end{conjecture}
We will verify Conjecture \ref{conj: in intro} in two examples from
the simplest ones, computing both sides of the equality independently.
One example is a tame action on a singular variety, the other a wild
non-linear action on a smooth variety. 

Keeping arguments above in mind, let us return to the linear case.
One difficulty in computing the right side of the equality in Conjecture
\ref{conj: linear McKay intro} is computing weights $\bw(E)$ explicitly,
and another is computing the moduli space $\GCov(D)$. Taking an equivariant
immersion $\sv\hookrightarrow V$ is useful also in solving the former
difficulty for some linear actions. With $E$, $F$ and $H$ as before,
if $V=\AA_{D}^{d}$ has a linear $G$-action and if $V_{0}:=V\otimes_{\cO_{D}}k$,
then the weight of $E$ with respect to $V$ is, by definition, 
\[
\bw_{V}(E)=\codim(V_{0}^{H},V_{0})-\bv_{V}(E)
\]
with another function $\bv_{V}$ on $\GCov(D)$ and $V_{0}^{H}$ the
$H$-fixed point locus in $V_{0}$. The first term, $\codim(V_{0}^{H},V_{0})$,
is easy to compute, while the second generally not. However, if $G$
acts on $V$ by permutations of coordinates, then $\bv_{V}(E)$ is
represented in terms of the discriminant \cite{Wood-Yasuda-I}: in
this situation, we can associate a degree $d$ cover $C\to D$ to
a $G$-cover $E\to D$ and 
\[
\bv_{V}(E)=\frac{d_{C/D}}{2}
\]
with $d_{C/D}$ the discriminant exponent of the cover $C\to D$.
We generalize this equality to hyperplanes in $V$ defined by a $G$-invariant
linear form. For simplicity, we consider the case where $\sv\subset V$
is defined by 
\[
x_{1}+\cdots+x_{d}=0
\]
with $x_{1},\dots,x_{d}$ coordinates of $V$. 
\begin{prop}[See Corollary \ref{cor:weight non-permutation} for a slightly more
general result]
Let $C=\bigsqcup_{j=1}^{l}C_{j}$ be the decomposition of $C$ into
the connected components. Then
\[
\bv_{\sv}(E)=\frac{d_{C/D}}{2}-\min\left\{ \left\lfloor \frac{d_{C_{j}/D}}{[C_{j}:D]}\right\rfloor \mid1\le j\le l\right\} 
\]
with $[C_{j}:D]$ the degree of $C_{j}\to D$.
\end{prop}
Using this and assuming a motivic version of Krasner's formula \cite{MR0225756}
for counting local field extensions, we explicitly compute 
\[
\int_{\GCov(D)}\LL^{-\bv_{3\sv}}\, d\tau\text{ and }\int_{\GCov(D)}\LL^{\bw_{3\sv}}\, d\tau,
\]
when $G=S_{4}$ acts on $V\cong\AA_{D}^{4}$ by the standard representation,
$\sv\subset V$ is given by $x_{1}+x_{2}+x_{3}+x_{4}=0$ and $3\sv$
is the direct sum of three copies of $\sv$. We find that the two
integrals are dual to each other. The same kind of duality was observed
in \cite{Wood-Yasuda-I} and will be discussed in \cite{Wood-Yasuda-II}
in more details. Obtained formulas for these integrals are thought
of as (motivic) variants of mass formulas for local Galois representations
\cite{MR2354798,MR2354797,MR2411405} with respect to weights coming
from a non-permutation representation.

From Section \ref{sec:Motivic-integration-and} to \ref{sec:The-McKay-correspondence-linear},
we review the theory in the linear case and finally formulate the
McKay correspondence for linear actions. Most materials here are not
new and found for instance in \cite{Yasuda:2013fk}, although arguments
are refined and adjusted to our purpose. In Section \ref{sec:The-McKay-correspondence-non-linear},
we formulate the McKay correspondence for non-linear actions. In Section
\ref{sec:Computing-boundaries}, we study how to determine the centered
log structure $\fv^{|F|,\nu}$ under some assumptions. In Sections
\ref{sec:A-tame-example} and \ref{sec:A-wild-example}, we will compute
non-linear examples. In Sections \ref{sec:stable hyperplanes} and
\ref{sec: example S4}, we treat hyperplanes in permutation representations.
We will end the paper with concluding remarks in Section \ref{sec:Concluding-remarks}.

\subsection{Acknowledgments}

I wish to thank Yusuke Nakamura and Shuji Saito for useful discussions
directing my interests to non-linear actions, Johannes Nicaise and
Julien Sebag for their kind answers to my questions on motivic integration
over formal schemes, and Melanie Wood for stimulating discussions
in our joint work.

\subsection{Convention and notation}

If $X$ is an affine scheme, $\cO_{X}$ denotes its coordinate ring.
By the same symbol $\cO_{X}$, we sometimes denote also the structure
sheaf on a scheme $X$. This abuse of notation would not cause any
problem. When a group $G$ acts on $X$ from left, then we suppose
that $G$ acts on $\cO_{X}$ from right: for $g\in G$, if $\phi_{g}:X\to X$
is the $g$-action on $X$, then $g$ acts on $\cO_{X}$ by the pull-back
of functions by $\phi_{g}$. Throughout the paper, we fix an affine
scheme $D$ with $\cO_{D}$ a complete discrete valuation field. We
denote the residue field of $\cO_{D}$ by $k$ and suppose that $k$
is algebraically closed. For an integral scheme $X$, we denote by
$K(X)$ its function field. If $X$ is affine, then $K(X)$ is the
fraction (quotient) field of the ring $\cO_{X}$. Again, by abuse
of notation, $K(X)$ also denotes the constant sheaf on $X$ associated
to the function field. For a $D$-scheme $X$, we denote by $X_{0}$
the special fiber with the reduced structure: $X_{0}:=(X\times_{D}\Spec k)_{\red}$.

\section{Motivic integration and stringy motifs\label{sec:Motivic-integration-and}}

In this section, we review the theories of motivic integration over
ordinary (untwisted) arcs and stringy invariants, mainly developed
in \cite{Kontsevich-motivic,MR1664700,MR1672108,MR1677693,MR1905024,MR2075915}.

\subsection{Centered log varieties}

We call an integral $D$-scheme $X$ a $D$\emph{-variety }if $X$
is flat, separated and of finite type over $D$ and $X$ is smooth
over $D$ at the generic point of $X$. For a $D$-variety $X$, we
denote the smooth locus of $X$ by $X_{\sm}$ and the regular locus
by $X_{\reg}$. 

Let $X$ be a normal $D$-variety. We can define the \emph{canonical
sheaf $\omega_{X}=\omega_{X/D}$ }of $X$ over $D$ as in \cite[page 8]{MR3057950}.
On $X_{\sm}$, the canonical sheaf is isomorphic to $\bigwedge^{d}\Omega_{X/D}$
with $d$ the relative dimension of $X$ over $D$. Therefore we can
think of $\omega_{X}$ as a subsheaf of $(\bigwedge^{d}\Omega_{X/D})\otimes K(X)$.
We define the \emph{canonical divisor} of $X$, denoted by $K_{X}=K_{X/D}$,
to be the linear equivalence class of Weil divisors corresponding
to $\omega_{X}$. 

A \emph{log $D$-variety }is a pair $(X,\Delta)$ of a normal $D$-variety
$X$ and a Weil $\QQ$-divisor $\Delta$ such that $K_{X}+\Delta$
is $\QQ$-Cartier. We call $\Delta$ the \emph{boundary }of the log
variety. The \emph{canonical divisor }of a log $D$-variety $(X,\Delta)$
is $K_{(X,\Delta)}:=K_{X}+\Delta$.

A \emph{centered log $D$-variety }is a triple $\mathfrak{X}=(X,\Delta,W)$
such that $(X,\Delta)$ is a log $D$-variety and $W$ is a closed
subset of $X_{0}$, where $X_{0}$ denotes the special fiber of the
structure morphism $X\to D$ with the reduced structure. We call $W$
the \emph{center }of $\fX$. We also say that $\fX$ is a \emph{centered
log structure} on $X$. For a centered log variety $\mathfrak{X}=(X,\Delta,W)$,
we define a \emph{canonical divisor }of $\fX$ as the one of $(X,\Delta)$:
\[
K_{\fX}:=K_{(X,\Delta)}=K_{X}+\Delta.
\]

sometimes, we identify a normal $\QQ$-Gorenstein ($K_{X}$ is $\QQ$-Cartier)
$D$-variety $X$ with the log $D$-variety $(X,0)$, and identify
a log $D$-variety $(X,\Delta)$ with the centered log $D$-variety
$(X,\Delta,X_{0})$:
\begin{equation}
\xymatrix{\left\{ \substack{\text{normal \ensuremath{\QQ}-Gorenstein}\\
\text{\ensuremath{D}-varieties}
}
\right\} \ar@{^{(}->}[r] & \left\{ \text{log \ensuremath{D}-varieties}\right\} \ar@{^{(}->}[r] & \left\{ \text{centered log \ensuremath{D}-varieties}\right\} \\
X\ar@{|->}[r] & (X,0)\\
 & (X,\Delta)\ar@{|->}[r] & (X,\Delta,X_{0})
}
\label{eq:identifications center log}
\end{equation}

\subsection{Crepant morphisms\label{sub:Crepant-morphisms}}

For centered log $D$-varieties $\fX=(X,\Delta,W)$ and $\fX'=(X',\Delta',W')$,
a \emph{morphism $f:\fX\to\fX'$} is just a morphism $f:X\to X'$
of the underlying varieties with $f(W)\subset W'$. We say that a
morphism $\fX\to\fX'$ is \emph{proper }or \emph{birational }if it
is so as the morphism $f:X\to X'$ of the underlying varieties. We
say that a morphism $f:\fX\to\fX'$ is \emph{crepant }if 
\[
f^{-1}(W')=W\text{ and }K_{\fX}=f^{*}K_{\fX'}.
\]
The right equality should be understood that for $r\in\ZZ_{>0}$ such
that $r(K_{X}+\Delta)$ and $r(K_{X'}+\Delta')$ are Cartier, we have
a natural isomorphism
\[
\omega_{X'}^{[r]}(r\Delta')\cong f^{*}\omega_{X}^{[r]}(r\Delta).
\]
Here $\omega_{X}^{[r]}(r\Delta)$ is the invertible sheaf which is
identical to $\omega_{X}^{\otimes r}(r\Delta)$ on $X_{\reg}$. We
adopt this convention throughout the paper. 

Given a generically \'{e}tale morphism $f:X\to X'$ of normal $D$-varieties,
a centered log structure $\fX'$ on $X'$ induces a unique centered
log structure $\fX$ on $X$ such that the morphism $f:\fX\to\fX'$
is crepant. Conversely, if $f:X\to X'$ is additionally proper, then
for each centered log structure $\fX$ on $X$, there exists at most
one centered log structure on $\fX'$ such that $f:\fX\to\fX'$ is
crepant. 
\begin{rem}
For our purpose, we may slightly weaken the assumptions in the definition
of crepant morphisms. For instance, concerning the equality $f^{-1}(W')=W$,
we only need this equality outside $X_{\sm}\setminus X_{\reg}$. This
is because the locus $X_{\sm}\setminus X_{\reg}$ does not contribute
to stringy motifs at all, which will be defined below. However, for
simplicity, we will cling to our definition as above.
\end{rem}

\subsection{Motivic integration\label{sub:Motivic-integration}}

Let $\fX=(X,\Delta,W)$ be a centered log $D$-variety. An \emph{arc}
of $\fX$ is a $D$-morphism $D\to X$ sending the closed point of
$D$ into $W$. The \emph{arc space} of $\fX$, denoted $J_{\infty}\fX$,
is a $k$-scheme parameterizing the arcs of $\fX$. We put $D_{n}:=\Spec\cO_{D}/\fm_{D}^{n+1}$
with $\fm_{D}$ the maximal ideal of $\cO_{D}$. An \emph{$n$-jet}
of $\fX$ is a $D$-morphism $D_{n}\to X$ sending the unique point
of $D_{n}$ into $W$. For each $n$, there exists a $k$-scheme $J_{n}\fX$
parametrizing $n$-jets of $\fX$. For $n\ge m$, we have natural
morphisms $J_{n}\fX\to J_{m}\fX$ and the arc space $J_{\infty}\fX$
is identified with the projective limit of $J_{n}\fX$, $n\ge0$ with
respect to these maps. We have the induce maps 
\[
\pi_{n}:J_{\infty}\fX\to J_{n}\fX.
\]
For $n<\infty$, $J_{n}\fX$ are of finite type over $k$. For a morphism
$f:\fY\to\fX$ and each $n\in\ZZ_{\ge0}\cup\{\infty\}$, there exists
a natural map
\[
f_{n}:J_{n}\fY\to J_{n}\fX.
\]

The arc space $J_{\infty}\fX$ has the so-called \emph{motivic measure},
denoted by $\mu_{J_{\infty}\fX}$. The measure takes values in some
(semi-)ring, say $\cR$, which is often a suitable modification of
the Grothendieck (semi-)ring of $k$-varieties. In this paper, we
will fix $\cR$ satisfying the following properties: denoting by $[T]$
the class of a $k$-variety $T$ in $\cR$, we have 
\begin{itemize}
\item for a bijective morphism $S\to T$, we have $[S]=[T]$ in $\cR$,
\item putting $\LL:=[\AA_{k}^{1}]$, we have all fractional powers $\LL^{a}$,
$a\in\QQ$ in $\cR$,
\item an infinite series $\sum_{i=1}^{\infty}[T_{i}]\LL^{a_{i}}$ with $\lim_{i\to\infty}\dim T_{i}+a_{i}=-\infty$
converges, 
\item for a morphism $f:S\to T$ and for $n\in\ZZ_{\ge0}$, if every fiber
of $f$ admits a homeomorphism from or to the quotient $\AA_{k}^{n}/G$
for some linear action of a finite group $G$ on $\AA_{k}^{n}$, then
$[S]=[T]\LL^{n}$. 
\end{itemize}
One possible choice is the field of Puiseux series in $t^{-1}$, 
\[
\cR:=\bigcup_{r=1}^{\infty}\ZZ((t^{-1/r})),
\]
where we put $[T]$ to be the Poincar\'{e} polynomial as in \cite{MR2770561}.

A subset $A\subset J_{\infty}\fX$ is called \emph{stable }if there
exists $n\in\ZZ_{\ge0}$ such that $\pi_{n}(A)\subset J_{n}\fX$ is
a constructible subset and $A=\pi_{n}^{-1}\pi_{n}(A)$ and for every
$m\ge n$, every fiber of the map $\pi_{n+1}(A)\to\pi_{n}(A)$ is
homeomorphic to $\AA_{k}^{n}$. The measure of a stable subset $A$
is given by 
\[
\mu_{J_{\infty}\fX}(A):=[\pi_{n}(A)]\LL^{-nd}\quad(n\gg0).
\]
More generally, we can define the measure for \emph{measurable subsets},
which are roughly the limits of stable subsets. 

Let $\Phi:C\to\cR\cup\{\infty\}$ be a measurable function on a subset
$C\subset J_{\infty}\fX$, that is, the image of $\Phi$ is countable,
all fibers $\Phi^{-1}(a)$ are measurable and $\mu_{J_{\infty}\fX}(\Phi^{-1}(\infty))=0$.
We define 
\[
\int_{C}\Phi\,\mu_{J_{\infty}\fX}:=\sum_{a\in\cR}\mu_{J_{\infty}\fX}(\Phi^{-1}(a))\cdot a\in\cR\cup\{\infty\}.
\]

\subsection{Stringy invariants}

We still suppose that $\mathfrak{X}=(X,\Delta,W)$ is a centered log
$D$-variety. 
\begin{defn}
\label{def: order function}To a coherent ideal sheaf $I\ne0$ on
$X$ defining a closed subscheme $Z\subsetneq X$, we associate the
\emph{order function}, 
\[
\ord\, I=\ord\, Z:J_{\infty}\fX\to\ZZ_{\ge0}\cup\{\infty\},
\]
as follows: for an arc $\gamma:D\to X$, the pull-back $\gamma^{-1}I$
of $I$ is an ideal of $\cO_{D}$ and of the form $\fm_{D}^{l}$ for
some $l\in\ZZ_{\ge0}\cup\{\infty\}$, where we put $(0):=\fm_{D}^{\infty}$
by convention. For a fractional ideal $I$ (that is, a coherent $\cO_{X}$-submodule
of $K(X)$), if we write $I=I_{+}\cdot I_{-}^{-1}$ for ideal sheaves
$I_{+}$ and $I_{-}$ with $I_{-}$ locally principal, then we put
\[
\ord\, I:=\ord\, I_{+}-\ord\, I_{-}.
\]
Here we put $\ord\, I=\infty$ if either $\ord\, I_{+}=\infty$ or
$\ord\, I_{-}=\infty$. Similarly, for a $\QQ$-linear combination
$Z=\sum_{i=1}^{n}a_{i}Z_{i}$ of closed subschemes $Z_{i}\subsetneq X$,
we define
\[
\ord\, Z:=\sum_{i=1}^{n}a_{i}\cdot\ord\, Z_{i},
\]
taking values in $\QQ\cup\{\infty\}$. \end{defn}
\begin{rem}
For a closed subscheme $Z\subsetneq X$, we expect that $(\ord\, Z)^{-1}(\infty)$
has measure zero. The author does not know if this has been proved,
but this follows from the change of variables formula, if there exists
a resolution of singularities $f:\tilde{X}\to X$ so that $\tilde{X}$
is regular and $\tilde{X}_{0}\cup f^{-1}(Z)$ is a simple normal crossing
divisor. If the expectation is actually true, then order functions
for fractional ideals and $\QQ$-linear combination of closed subschemes
are well-defined modulo measure zero subsets. 
\end{rem}
Let $r\in\ZZ_{>0}$ be such that $rK_{\fX}$ is Cartier. Since the
sheaf $\cO_{X}(rK_{\fX})=\omega_{X}^{[r]}(r\Delta)$ is invertible
and thought of as a subsheaf of the constant sheaf $(\bigwedge^{d}\Omega_{X/D})^{\otimes r}\otimes K(X)$,
we can define a fractional ideal sheaf $I_{\fX}^{r}$ by the equality
of subsheaves of $(\bigwedge^{d}\Omega_{X/D})^{\otimes r}\otimes K(X)$,
\[
\left(\bigwedge^{d}\Omega_{X/D}\right)^{\otimes r}/\tors=I_{\fX}^{r}\cdot\cO_{X}(rK_{\fX}).
\]
We then put a function $\bf_{\fX}$ on $J_{\infty}\fX$ by
\[
\bf_{\fX}:=\frac{1}{r}\ord\, I_{\fX}^{r}.
\]
Since $(I_{\fX}^{r})^{n}=I_{\fX}^{r\cdot n}$, the function $\bf_{\fX}$
is independent of the choice of $r$. $ $If $X$ is smooth, then
we simply have $\bf_{\fX}=\ord\,\Delta$. 
\begin{defn}
The \emph{stringy motif }of $\fX$ is defined to be
\[
M_{\st}(\fX):=\int_{J_{\infty}\fX}\LL^{\bf_{\fX}}\, d\mu_{J_{\infty}\fX}.
\]
We also write $M_{\st}(\fX)=M_{\st}(X,\Delta)_{W}$ and sometimes
omit $\Delta$ if $\Delta=0$, and $W$ if $W=X_{0}$. When the integral
above converges, we call $\fX$ \emph{stringily log terminal. }When
diverges, we put $M_{\st}(\fX):=\infty$. \end{defn}
\begin{conjecture}
\label{conj: crepant equal stringy}If a morphism $f:\fX\to\fX'$
of centered log $D$-varieties is proper, birational and crepant,
then 
\[
M_{\st}(\fX)=M_{\st}(\fX').
\]
\end{conjecture}
\begin{prop}
\label{prop:crepant equal stringy}Conjecture \ref{conj: crepant equal stringy}
holds if there exists a proper birational morphism $Y\to X$ of $D$-varieties
such that $Y\otimes_{\cO_{D}}K(D)$ is smooth over $K(D)$. In particular,
Conjecture \ref{conj: crepant equal stringy} holds if $K(D)$ has
characteristic zero.\end{prop}
\begin{proof}
Let $X_{\eta}$ be the generic fiber of $X\to D$. From the Hironaka
theorem, there exists a coherent ideal sheaf $I_{\eta}\subset\cO_{X_{\eta}}$
such that the blowup of $X_{\eta}$ along $I_{\eta}$ is smooth over
$K(D)$. Let $I\subset\cO_{X}$ be an coherent ideal sheaf such that
$I|_{X_{\eta}}=I_{\eta}$. The blowup of $X$ along $I$ has a smooth
generic fiber. Hence the second assertion of the lemma follows from
the first one.

To prove the first one, we can apply the version of the change of
variables formula proved by Sebag \cite[Th. 8.0.5]{MR2075915} (see
also \cite{MR2885336}). If $\fY$ is the centered log structure on
$Y$ such that the induced morphism $f:\fY\to\fX$ is crepant, then
the change of variables formula shows that
\[
\int_{J_{\infty}\fX}\LL^{\bf_{\fX}}\, d\mu_{J_{\infty}\fX}=\int_{J_{\infty}\fY}\LL^{\bf_{\fX}\circ f_{\infty}-\ord\,\mathrm{jac}_{f}}\, d\mu_{J_{\infty}\fY}.
\]
Here $\ord\,\mathrm{jac}_{f}$ is the function of Jacobian orders
as defined in \cite[page 29]{MR2075915}.$ $ For $r\in\ZZ_{>0}$
such that $rK_{\fX}$ and $rK_{\fY}$ are Cartier, we have
\begin{align*}
\left(f^{*}\left(\bigwedge^{d}\Omega_{X/D}\right)^{\otimes r}\right)/\tors & =f^{-1}I_{\fX}^{r}\cdot\cO_{\fY}(rK_{\fY})\text{ and}\\
\left(\bigwedge^{d}\Omega_{Y/D}\right)^{\otimes r}/\tors & =I_{\fY}^{r}\cdot\cO_{\fY}(rK_{\fY}).
\end{align*}
This shows that 
\[
\bf_{\fX}\circ f_{\infty}-\ord\,\mathrm{jac}_{f}=\bf_{\fY}.
\]
We obtain $M_{\st}(\fX)=M_{\st}(\fY)$, and similarly $M_{\st}(\fX')=M_{\st}(\fY)$.
We have proved the proposition. \end{proof}
\begin{prop}
\label{prop: ordinary explicit formula}Let $\fX=(X,\Delta,W)$ be
a centered log $D$-variety and write 
\[
\Delta=\sum_{h=1}^{l}a_{h}A_{h}+\sum_{i=1}^{m}b_{i}B_{i}+\sum_{j=1}^{n}c_{j}C_{j}\quad(a_{h},b_{i}\in\QQ,\, c_{j}\in\QQ\setminus\{0\})
\]
such that $A_{h}$ are the irreducible components of the closure of
$X_{0}\cap X_{\sm}$, $B_{i}$ are the irreducible components of $X_{0}\setminus X_{\sm}$
and $C_{j}$ are prime divisors not contained in $X_{0}$. Let

\begin{gather*}
A_{h}^{\circ}:=A_{h}\cap X_{\sm}=A_{h}\setminus\left(\overline{X_{0}\setminus A_{h}}\right)\text{ and}\\
C_{J}^{\circ}:=\bigcap_{j\in J}C_{j}\setminus\bigcup_{j\in\{1,..,n\}\setminus J}C_{j},
\end{gather*}
with $\overline{X_{0}\setminus A_{h}}$ the closure of $X_{0}\setminus A_{h}$.
We suppose that $X$ is regular and that $\bigcup_{j=1}^{n}C_{j}$
is simple normal crossing, that is, for any $J\subset\{1,\dots,n\},$
the scheme-theoretic intersection $\bigcap_{j\in J}C_{j}$ is smooth
over $D$. Then $\fX$ is stringily log terminal if and only if $c_{j}<1$
for every $j$ with $C_{j}\cap W\cap X_{\sm}\ne\emptyset$. Moreover,
if it is the case, 
\[
M_{\st}(\fX)=\sum_{h=1}^{l}\LL^{a_{h}}\sum_{J\subset\{1,\dots,n\}}[W\cap A_{h}^{\circ}\cap C_{J}^{\circ}]\prod_{j\in J}\frac{\LL-1}{\LL^{1-c_{j}}-1}.
\]
\end{prop}
\begin{proof}
We first note that the locus $X_{0}\setminus X_{\sm}$ does not have
any arc, hence not contribute to $M_{\st}(\fX)$. Since 
\[
X_{0}\cap X_{\sm}=\bigsqcup_{h=1}^{l}A_{h}^{\circ},
\]
we can decompose $M_{\st}(\fX)$ into the sum of components corresponding
to $A_{h}^{\circ}$, $h=1,..,l$. The divisor $a_{h}A_{h}$ contributes
to the component corresponding to $A_{h}^{\circ}$ by the multiplication
with $\LL^{a_{h}}$. From all these arguments, the proposition has
been reduced to the formula
\[
M_{\st}(\fX)=\sum_{J\subset\{1,\dots,n\}}[W\cap C_{J}^{\circ}]\prod_{j\in J}\frac{\LL-1}{\LL^{1-c_{j}}-1}
\]
in the case where $X$ is smooth and $\Delta=\sum_{j}c_{j}C_{j}$.
This is the standard explicit formula (see for instance \cite{MR1672108}). 
\end{proof}

\subsection{Group actions}
\begin{defn}
A \emph{centered log $G$-$D$-variety }is a centered log $D$-variety
$\fX=(X,\Delta,W)$ endowed with a faithful $G$-action on $X$ such
that $\Delta$ and $W$ are stable under the $G$-action. Given a
variety $X$ with a faithful $G$-action, we say that a centered log
structure $\fX$ on $X$ is $G$\emph{-equivariant }if it is a centered
log $G$-$D$-variety.
\end{defn}
For a centered log $G$-$D$-variety $\fX$, the arc space $J_{\infty}\fX$
has a natural $G$-action. We define a motivic measure on $(J_{\infty}\fX)/G$,
denoted by $\mu_{(J_{\infty}\fX)/G}$, in the same way as defining
the motivic measure on $J_{\infty}\fX$ except that in the definition
of stable subsets, say $A$, fibers of $\pi_{n+1}(A)\to\pi_{n}(A)$
are only assumed to be homeomorphic the quotient $\AA_{k}^{d}/H$
for some linear action of a finite group $H$ on $\AA_{k}^{d}$. 

The function $\bf_{\fX}$ on $J_{\infty}\fX$ is $G$-invariant and
gives a function on $(J_{\infty}\fX)/G$, which we again denote by
$\bf_{\fX}$. We define 
\[
M_{\st,G}(\fX):=\int_{(J_{\infty}\fX)/G}\LL^{\bf_{\fX}}\, d\mu_{(J_{\infty}\fX)/G}.
\]
The reader should not confuse $M_{\st,G}(\fX)$ with the orbifold
stringy motif $M_{\st}^{G}(\fX)$, defined later. 

Let us define a $G$-\emph{prime divisor} as a divisor of the form
$\sum_{i=1}^{l}D_{i}$, where $D_{i}$ are prime divisors permuted
transversally by the $G$-actions, 
\begin{prop}
\label{prop:equivariant explicit formula}Let $\fX=(X,\Delta,W)$
be a centered log $G$-$D$-variety and write 
\[
\Delta=\sum_{h=1}^{l}a_{h}A_{h}+\sum_{i=1}^{m}b_{i}B_{i}+\sum_{j=1}^{n}c_{j}C_{j}\quad(a_{h},b_{i}\in\QQ,\, c_{j}\in\QQ\setminus\{0\})
\]
such that $A_{h}$ are the distinct $G$-prime divisors such that
$\bigcup A_{h}$ is equal to the closure of $X_{0}\cap X_{\sm}$,
$B_{i}$ are the distinct $G$-prime divisors with $\bigcup_{i}B_{i}=X_{0}\setminus X_{\sm}$
and $C_{j}$ are $G$-prime divisors not contained in $X_{0}$. We
suppose 
\begin{itemize}
\item $X$ is regular, 
\item \label{enu:snc}for any $J\subset\{1,\dots,n\},$ the scheme-theoretic
intersection $\bigcap_{j\in J}C_{j}$ is smooth over $D$, and
\item for every $j$ with $C_{j}\cap W\cap X_{\sm}\ne\emptyset$, $c_{j}<1$. 
\end{itemize}
With the same notation as in Proposition \ref{prop: ordinary explicit formula},
we have 
\[
M_{\st,G}(\fX)=\sum_{h=1}^{l}\LL^{a_{h}}\sum_{J\subset\{1,\dots,n\}}\left[\frac{W\cap A_{h}^{\circ}\cap C_{J}^{\circ}}{G}\right]\prod_{j\in J}\frac{\LL-1}{\LL^{1-c_{j}}-1}.
\]

\end{prop}

\section{$G$-arcs}

In the last subsection, we considered motivic integration over varieties
endowed with finite group actions. However we considered only ordinary
(untwisted) arcs, which are not general enough to apply to the McKay
correspondence. Suitably generalized arcs were introduced by Denef
and Loeser \cite{MR1905024} in characteristic zero. The author \cite{Yasuda:2013fk}
further generalized them to arbitrary characteristics. We may use
generalized arcs of orbifolds or Deligne-Mumford stacks as in \cite{MR2069013,MR2027195,MR2271984,Yasuda:2013fk}
so that we can treat general orbifolds, having group actions only
locally. We do not pursue generalization in this direction, however.

From now on, we fix a finite group $G$. 
\begin{defn}
By a $G$-\emph{cover }of $D$, we mean a $D$-scheme $E$ endowed
with a left $G$-action such that $E\otimes_{\cO_{D}}K(D)$ is an
\'{e}tale $G$-torsor over $\Spec K(D)$ and $E$ is the normalization
of $D$ in $\cO_{E\otimes_{\cO_{D}}K(D)}$. We denote by $\GCov(D)$
the set of $G$-covers of $D$ up to isomorphism. \end{defn}
\begin{rem}
In the tame case, there is a one-to-one correspondence between the
points of $\GCov(D)$ and the conjugacy classes in $G$. In the wild
case, however, $\GCov(D)$ is expected to be an infinite dimensional
space having a countable stratification with finite-dimensional strata.
\end{rem}
We now fix the following notation: $E$ is a $G$-cover of $D$, $F$
is a connected component of $E$ with a stabilizer $H$ so that $F$
is an $H$-cover of $D$. 
\[
\xymatrix{E\ar[dr]_{G\text{-cover}} &  & F\ar@{_{(}->}_{\text{conn. comp.}}[ll]\ar[dl]^{H\text{-cover}}\\
 & D
}
\]

\begin{lem}
Let $\Aut(E)$ be the automorphism group of $E$ as a $G$-cover of
$D$, that is, it consists of $G$-equivariant $D$-automorphisms
of $E$. We have a natural isomorphism 
\[
\Aut(E)\cong C_{G}(H)^{\op},
\]
where the right side is the opposite group of the centralizer of $H$
in $G$. \end{lem}
\begin{proof}
If $E$ is the trivial $G$-cover $D\times G$ of $D$, then its automorphisms
are nothing but the right $G$-action on $G$. Therefore $\Aut(E)=G^{\op}$. 

For the general case, let $E_{F}$ be the normalization of the fiber
product $E\times_{D}F$. This is a trivial $G$-cover of $F$ and
we have a natural injection 
\[
\Aut(E)\to\Aut(E_{F})=G^{\op}.
\]
Its image is the automorphisms of $E_{F}$ compatible with the action
of $\Gal(F/D)=H$. This shows the lemma. 
\end{proof}
Let $V$ be a $D$-variety endowed with a faithful left $G$-action. 
\begin{defn}
We define an \emph{$E$-twisted $G$-arc} of $V$ as a $G$-equivariant
$D$-morphism $E\to V$ and a $G$\emph{-arc }of $V$ as an $E$-twisted
$G$-arc for some $E$. Two $G$-arcs $E\to V$ and $E'\to V$ are
said to be \emph{isomorphic} if there exists a $G$-equivariant $D$-isomorphism
$E\to E'$ compatible with the morphisms to $V$. We denote by $J_{\infty}^{G,E}V$
the set of isomorphism classes of $E$-twisted $G$-arcs of $V$ and
by $J_{\infty}^{G}V$ the set of isomorphism classes of $G$-arcs
of $V$.
\end{defn}
Obviously, 
\[
J_{\infty}^{G}V=\bigsqcup_{E\in\GCov(D)}J_{\infty}^{G,E}V.
\]
Let $\Hom_{D}^{G}(E,V)$ be the space of $G$-equivariant $D$-morphisms
$E\to V$. We define a left action of $C_{G}(H)=\Aut(E)^{\op}$ on
$\Hom_{D}^{G}(E,V)$ as follows: for $a\in C_{G}(H)=\Aut(E)^{\op}$
and $f\in\Hom_{D}^{G}(E,V)$, 
\begin{equation}
a\cdot(E\xleftarrow{f}V):=(V\xleftarrow{f}E\xleftarrow{a}E)=(V\xleftarrow{a}V\xleftarrow{f}E).\label{eq:action on Xi}
\end{equation}
By definition, we have 
\begin{equation}
J_{\infty}^{G,E}V=\Hom_{D}^{G}(E,V)/C_{G}(H).\label{eq:J =00003D Hom / ?}
\end{equation}

For $n\in\ZZ_{\ge0}$, we put $F_{n}:=F/\fm_{F}^{n\cdot h+1}$ with
$h:=\sharp H$ and define $E_{n}:=\bigcup_{g\in G}g(F_{n})$. In particular,
$F_{0}\cong\Spec k$ and $E_{0}$ consists of the closed points of
$E$ with the reduced scheme structure. 
\begin{defn}
We define an $E$-\emph{twisted $G$-n-jet }of $V$ as a $G$-equivariant
$D$-morphism $E_{n}\to V$ and put
\begin{gather*}
J_{n}^{G,E}V:=\Hom_{D}^{G}(E_{n},V)/C_{G}(H)\text{ and}\\
J_{n}^{G}V=\bigsqcup_{E\in\GCov(D)}J_{n}^{G,E}V.
\end{gather*}
Here the $C_{G}(H)$-action on $\Hom_{D}^{G}(E_{n},V)$ is similarly
defined as (\ref{eq:action on Xi}).
\end{defn}
Note that if $E\not\cong E'$, then $E$-twisted and $E'$-twisted
$G$-$n$-jets never give the same point of $J_{n}^{G}V$. For each
$n\in\ZZ_{\ge0}$ and $E\in\GCov(D)$, we have natural maps
\[
J_{\infty}^{G,E}V\to J_{n}^{G,E}V\text{ and }J_{\infty}^{G}V\to J_{n}^{G}V,
\]
both of which we will denote by $\pi_{n}$. We have obtained the following
commutative diagram:

\[
\xymatrix{J_{\infty}^{G,E}V\ar[r]^{\pi_{n+1}}\ar[d] & J_{n+1}^{G,E}V\ar[r]\ar[d] & J_{n}^{G,E}V\ar[r]\ar[d] & \{E\}\ar[d]\\
J_{\infty}^{G}V\ar[r]_{\pi_{n+1}} & J_{n+1}^{G}V\ar[r] & J_{n}^{G}V\ar[r] & \GCov(D)
}
\]

\begin{rem}
\label{rem: finite-dimensional strata}In \cite{Yasuda:2013fk}, the
author conjectured that the sets $\GCov(D)$, $J_{n}^{G,E}V$, $J_{n}^{G}V$
$(0\le n<\infty)$ are realized as $k$-schemes admitting stratifications
with at most countably many finite-dimensional strata, which will
be necessary below to define the motivic measure. 
\end{rem}
Let $X$ be the quotient scheme $V/G$, writing the quotient morphism
as 
\[
p:V\to X.
\]
Given a $G$-arc $E\to V$, we get an arc $D\to X$ by taking the
$G$-quotients of the source and the target. This gives a natural
map
\[
p_{\infty}:J_{\infty}^{G}V\to J_{\infty}X.
\]
Let $T\subset V$ be the ramification locus of $\pi$ say with the
reduced scheme structure and $\bar{T}\subset X$ its image. The map
$p_{\infty}$ restricts to the bijection
\[
J_{\infty}^{G}V\setminus J_{\infty}^{G}T\to J_{\infty}X\setminus J_{\infty}\bar{T}.
\]

For $n<\infty$, we have a natural map 
\[
p_{n}:\pi_{n}(J_{\infty}^{G}V)\to J_{n}X,
\]
where $\pi_{n}$ denotes the natural map $J_{\infty}^{G}V\to J_{n}^{G}V$. 

For a centered log $G$-$D$-variety $\fV$ and $n\in\ZZ_{\ge0}\cup\{\infty\}$,
we define $J_{n}^{G}\fV$ and $J_{n}^{G,E}\fV$ as the subsets of
$J_{n}^{G}V$ and $J_{n}^{G,E}V$ consisting of the morphisms $E_{n}\to V$
sending the closed points of $E_{n}$ into the center of $\fV$.

\section{The untwisting technique revisited\label{sec:The-untwisting-technique}}

In this section, we revisit the technique of \emph{untwisting, }which
was first used by Denef and Loeser \cite{MR1905024} in characteristic
zero and generalized to arbitrary characteristics by the author \cite{Yasuda:2013fk}.
Our constructions below are slightly different and refined from the
ones in \cite{Yasuda:2013fk}.

Let us now turn to the case where $V$ is an affine space over $D$
and the given $G$-action is linear. We keep to fix a $G$-cover $E$
of $D$ and a connected component $F$ of $E$ with stabilizer $H$. 

For a free $\cO_{D}$-module $M$ of rank $d$, let $\cO_{V}:=S_{\cO_{D}}^{\bullet}M$
be its symmetric algebra and put 
\[
V=\Spec\cO_{V}=\AA_{D}^{d}.
\]
We suppose that the module $M$ and hence the $\cO_{D}$-algebra $\cO_{V}$
have faithful \emph{right} $G$-actions. Then $V$ has the induced
\emph{left} $G$-action. The set $\Hom_{D}^{G}(E,V)$ can be identified
with the $\cO_{D}$-module 
\begin{gather*}
\Xi_{F}:=\Hom_{\cO_{D}}^{H}(M,\cO_{F})=\Hom_{\cO_{D}}^{G}(M,\cO_{E}).
\end{gather*}
We call $\Xi_{F}$ the \emph{tuning module}. 
\begin{rem}
If we fix a basis of $M$, then the module $\Hom_{\cO_{D}}(M,\cO_{E})$
is identified with $\cO_{E}^{\oplus d}$. This module $\cO_{E}^{\oplus d}$
has two $G$-actions: one is the diagonal $G$-action induced from
the given $G$-action on $\cO_{E}$ and the other is the one induced
from the $G$-action on $M$. For an element of $\cO_{E}^{\oplus d}$
corresponding to a $G$-equivariant map $M\to\cO_{E}$, two actions
must coincide. We thus can identify $\Xi_{F}$ with the locus in $\cO_{E}^{\oplus d}$
where the two actions coincide. This was how the module $\Xi_{F}$
was presented in previous papers \cite{Yasuda:2013fk,Wood-Yasuda-I}.\end{rem}
\begin{lem}[\cite{Yasuda:2013fk,Wood-Yasuda-I}]
The module $\Xi_{F}$ is a free $\cO_{D}$-module of rank $d$. Moreover
it is a saturated $\cO_{D}$-submodule of $\Hom_{\cO_{D}}(M,\cO_{F})$
and of $\Hom_{\cO_{D}}(M,\cO_{E})$: for $a\in\cO_{D}$ and $f\in\Hom_{\cO_{D}}(M,\cO_{E})$,
if $af\in\Xi_{F}$, then $f\in\Xi_{F}$. 
\end{lem}
From (\ref{eq:J =00003D Hom / ?}), 
\[
J_{\infty}^{G,E}V=\Xi_{F}/C_{G}(H)
\]
Note that the $C_{G}(H)$-action on $\Xi_{F}$ is $\cO_{D}$-linear. 
\begin{lem}
The maps 
\[
\pi_{n+1}(J_{\infty}^{G}V)\to\pi_{n}(J_{\infty}^{G}V)
\]
have fibers homeomorphic to the quotient of $\AA_{k}^{d}$ by a linear
action of a finite group. \end{lem}
\begin{proof}
If we denote the map $\Xi_{F}\to\Hom_{D}^{H}(F_{n},V)$ again by $\pi_{n}$,
the image $\pi_{n}(\Xi_{F})$ is isomorphic to $(\cO_{D}/\fm_{D}^{n+1})^{\oplus d}$.
This shows that the fibers of 
\[
\pi_{n+1}(\Xi_{F})\to\pi_{n}(\Xi_{F})
\]
are isomorphic to $\AA_{k}^{d}$, proving the lemma. \end{proof}
\begin{defn}
We define a motivic measure $\mu_{J_{\infty}^{G}V}$ on $J_{\infty}^{G}V$
in the same way as the ones on $J_{\infty}V$ and $(J_{\infty}V)/G$.
If $\fV$ is a $G$-equivariant centered log structure on $V$, we
define the measure $\mu_{J_{\infty}^{G}\fV}$ on $J_{\infty}^{G}\fV$
as the restriction of $\mu_{J_{\infty}^{G}V}$. \end{defn}
\begin{rem}
For the definition above making sense, we need the conjecture that
moduli spaces $\GCov(D)$ and $J_{n}^{G}V$ exist and have some finiteness
(see Remark \ref{rem: finite-dimensional strata}).\end{rem}
\begin{defn}
We define modules, 
\begin{gather*}
M^{|F|}:=\Hom_{\cO_{D}}(\Xi_{F},\cO_{D})\text{ and}\\
M^{\left\langle F\right\rangle }:=M^{|F|}\otimes_{\cO_{D}}\cO_{F}=\Hom_{\cO_{D}}(\Xi_{F},\cO_{F}),
\end{gather*}
which are free modules of rank $d$ over $\cO_{D}$ and $\cO_{F}$
respectively. We define an $\cO_{D}$-linear map 
\begin{align*}
u^{*}=u_{F}^{*}:M & \to M^{\left\langle F\right\rangle }\\
m & \mapsto(\Xi\ni f\mapsto f(m)\in\cO_{F}),
\end{align*}
identifying $\Xi_{F}$ with $\Hom_{\cO_{D}}^{H}(M,\cO_{F})$ rather
than $\Hom_{\cO_{D}}^{G}(M,\cO_{E})$. \end{defn}
\begin{lem}
\label{lem:equivariant in two ways}We suppose that $H$ and $C_{G}(H)$
acts on $M$ by restricting the given $G$-action.
\begin{enumerate}
\item With respect to the $H$-action on $M^{\left\langle F\right\rangle }$
induced from the $H$-action on $\cO_{F}$, the map $u^{*}$ is $H$-equivariant.
\item With respect to the $C_{G}(H)$-action on $M^{\left\langle F\right\rangle }$
induced from the (left) $C_{G}(H)$-action on $\Xi_{F}$, the map
$u^{*}$ is $C_{G}(H)$-equivariant. 
\end{enumerate}
\end{lem}
\begin{proof}

\begin{enumerate}
\item For $h\in H$ and $m\in M$, we have 
\begin{align*}
u^{*}(mh) & =(f\mapsto f(mh))\\
 & =(f\mapsto f(m)h)\\
 & =(f\mapsto f(m))h,
\end{align*}
since $f\in\Xi$ are $H$-equivariant. This shows the assertion.
\item Let $M^{\left\langle E\right\rangle }:=\Hom_{\cO_{D}}(\Xi_{F},\cO_{E})$
and consider the natural map
\[
u_{E}^{*}:M\to M^{\left\langle E\right\rangle },\, m\mapsto(f\mapsto f(m)),
\]
now identifying $\Xi_{F}$ with $\Hom_{\cO_{D}}^{G}(M,\cO_{E})$.
This map is $C_{G}(H)$-equivariant. Indeed, for $g\in C_{G}(H)$
and $m\in M$, from (\ref{eq:action on Xi}), we have
\begin{align*}
u_{E}^{*}(mg) & =(f\mapsto f(mg))\\
 & =(f\mapsto f(m)g)\\
 & =(f\mapsto(gf)(m)).
\end{align*}
 The map $u_{E}^{*}$ factors as
\[
M\xrightarrow{u_{F}^{*}}M^{\left\langle F\right\rangle }\hookrightarrow M^{\left\langle E\right\rangle }.
\]
Since the inclusion $M^{\left\langle F\right\rangle }\hookrightarrow M^{\left\langle E\right\rangle }$
is also $C_{G}(H)$-equivariant, so does $u_{F}^{*}$. 
\end{enumerate}
\end{proof}
Note that the $H$- and $C_{G}(H)$- actions above on $M^{\left\langle F\right\rangle }$
commute.
\begin{defn}
We define the \emph{untwisting variety }(resp. \emph{pre-untwisting)
variety }of $V$ with respect to $F$ as 
\[
V^{|F|}:=\Spec S_{\cO_{D}}^{\bullet}M^{|F|}=\AA_{D}^{d}\text{ (resp. }V^{\left\langle F\right\rangle }:=\Spec S_{\cO_{F}}^{\bullet}M^{\left\langle F\right\rangle }=\AA_{F}^{d}).
\]
We denote the projection $V^{\left\langle F\right\rangle }\to V^{|F|}$
by $r=r_{F}$, $r$ standing for the restriction of scalars (see diagram
(\ref{eq:fundamental square}) below). The map $u^{*}$ defines a
$D$-morphism
\[
u:V^{\left\langle F\right\rangle }\to V,
\]
which is both $H$- and $C_{G}(H)$-equivariant. We call the pair
of $r$ and $u$ the \emph{untwisting correspondence }of $V$ with
respect to $F$. 
\end{defn}
Let $X:=V/G$ and identify $\cO_{X}$ with $(\cO_{V})^{G}$. Since
the $H$-invariant subring of $\cO_{V^{\left\langle F\right\rangle }}$
is 
\[
(\cO_{V^{\left\langle F\right\rangle }})^{H}=\cO_{V^{|F|}},
\]
we have
\[
u^{*}(\cO_{X})\subset\cO_{V^{|F|}}.
\]
We denote the induced morphism $V^{|F|}\to X$ by $p^{|F|}$. We have
the following commutative diagram: 
\begin{equation}
\xymatrix{ & V^{\left\langle F\right\rangle }=\AA_{F}^{d}\ar[dl]_{u}\ar[dr]^{r}\\
V=\AA_{D}^{d}\ar[dr]_{p} &  & V^{|F|}=\AA_{D}^{d}\ar[dl]^{p^{|F|}}\\
 & X=V/G
}
\label{eq:fundamental square}
\end{equation}

\begin{lem}

\begin{enumerate}
\item The map 
\begin{eqnarray*}
\Hom_{F}^{H}(F,V^{\left\langle F\right\rangle }) & \to & \Xi_{F}=\Hom_{D}^{H}(F,V)\\
\gamma & \mapsto & u\circ\gamma
\end{eqnarray*}
is bijective. 
\item The map 
\[
\Hom_{F}^{H}(F,V^{\left\langle F\right\rangle })\to J_{\infty}V^{|F|}=\Hom_{D}(D,V^{|F|})
\]
sending a morphism $F\to V^{\left\langle F\right\rangle }$ to the
induced one of quotients, 
\[
D=F/H\to V^{|F|}=V^{\left\langle F\right\rangle }/H,
\]
is bijective.
\end{enumerate}
\end{lem}
\begin{proof}

\begin{enumerate}
\item With the identification, 
\[
\Hom_{F}^{H}(F,V^{\left\langle F\right\rangle })=\Hom_{\cO_{F}}^{H}(\Hom_{\cO_{D}}(\Xi_{F},\cO_{F}),\cO_{F}).
\]
the map of the assertion is identified with the map
\begin{align*}
a:\Hom_{\cO_{F}}^{H}(\Hom_{\cO_{D}}(\Xi_{F},\cO_{F}),\cO_{F}) & \to\Xi_{F}\\
\phi & \mapsto(m\mapsto\phi((f\mapsto f(m)))),
\end{align*}
where $m\in M$ and $f\in\Xi_{F}$. Let us consider the map
\begin{align*}
b:\Xi_{F} & \to\Hom_{\cO_{F}}^{H}(\Hom_{\cO_{D}}(\Xi_{F},\cO_{F}),\cO_{F})\\
f & \mapsto(z\mapsto z(f)),
\end{align*}
where $z\in\Hom_{\cO_{D}}(\Xi_{F},\cO_{F})$. The composition $a\circ b$
sends $f\in\Xi_{F}$ to 
\begin{align*}
 & \left(m\mapsto\left(z\mapsto z(f)\right)\left(h\mapsto h(m)\right)\right)\\
 & =\left(m\mapsto f(m)\right)\\
 & =f,
\end{align*}
and hence is the identity map. It follows that $a$ is surjective.
Now the assertion follows from the fact that the source and target
of $a$ are free $\cO_{D}$-modules of the same rank and $a$ is a
homomorphism of $\cO_{D}$-modules. 
\item We can give the converse by the base change associated to $F\to D$. 
\end{enumerate}
\end{proof}
In summary, we have a one-to-one correspondence between $\Xi_{F}$
and $J_{\infty}V^{|F|}$, induced from the untwisting correspondence.
From Lemma \ref{lem:equivariant in two ways}, the correspondence
is compatible with the $C_{G}(H)$-actions on both sides. Therefore
it descends to a one-to-one correspondence between $J_{\infty}^{G,E}V$
and $(J_{\infty}V^{|F|})/C_{G}(H)$. We obtain the following commutative
diagram: 
\begin{equation}
\xymatrix{ & \Hom_{F}^{H}(F,V^{\preuntwisting})\ar@{<->}[dl]_{\text{1-to-1}}\ar@{<->}[dr]^{\text{\text{1-to-1}}}\\
\Xi_{F}\ar[d]\ar@{<->}[rr]^{\text{1-to-1}} &  & J_{\infty}V^{|F|}\ar[d]\\
J_{\infty}^{G,E}V\ar[dr]_{p_{\infty}}\ar@{<->}[rr]^{\text{1-to-1}} &  & (J_{\infty}V^{|F|})/C_{G}(H)\ar[dl]\\
 & J_{\infty}X
}
\label{pentagon}
\end{equation}
For $n<\infty$, we have a similar diagram:
\begin{equation}
\xymatrix{ & \pi_{n}(\Hom_{F}^{H}(F,V^{\preuntwisting}))\ar@{->>}[dl]_{\beta}\ar@{<->}[dr]^{\text{\text{1-to-1}}}\\
\pi_{n}(\Xi_{F})\ar[d] &  & J_{n}V^{|F|}\ar[d]\\
\pi_{n}(J_{\infty}^{G,E}V)\ar[dr] &  & (J_{n}V^{|F|})/C_{G}(H)\ar[dl]\\
 & J_{n}X
}
\label{pentagon-1}
\end{equation}
Note that the arrow $\beta$ has no longer bijective. When $n=0$,
the diagram is represented as:
\begin{equation}
\xymatrix{ & V_{0}^{\left\langle F\right\rangle }\ar@{->>}[dl]_{\beta}\ar@{<->}[dr]^{\text{\text{1-to-1}}}\\
(V_{0})^{H}\ar[d] &  & V_{0}^{|F|}\ar[d]\\
(V_{0})^{H}/C_{G}(H)\ar[dr] &  & (V_{0}^{|F|})/C_{G}(H)\ar[dl]\\
 & X_{0}
}
\label{pentagon-1-1}
\end{equation}
Here $(V_{0})^{H}$ is the fixed-point locus of the $H$-action on
$V_{0}$.

\section{The change of variables formula}

The untwisting technique, discussed in the last section, enables us
to deduce a conjectural change of variables formula for the map $p_{\infty}:J_{\infty}^{G}V\to J_{\infty}X$.
In turn, it will derive the McKay correspondence for linear actions
in the next section. 

We keep the notation from the last section. 
\begin{defn}
Let $f:T\to S$ be a morphism of $D$-varieties which is generically
\'{e}tale. The \emph{Jacobian ideal (sheaf)} 
\[
\Jac_{f}=\Jac_{T/S}\subset\cO_{T}
\]
is defined as the $0$-th Fitting ideal (sheaf) of $\Omega_{T/S}$,
the sheaf of K\"{a}hler differentials. We denote by $\bj_{f}$ the
order function of $\Jac_{f}$ on $J_{\infty}T$, $(J_{\infty}T)/G$
or $J_{\infty}^{G}T$ if $T$ has a faithful action of a finite group
$G$. The ambiguity of the domain will not cause a confusion.\end{defn}
\begin{rem}
When $T$ is smooth, the function $\bj_{f}$ on $J_{\infty}T$ coincides
with the Jacobian order function, denoted by $\ord\,\mathrm{jac}_{f}$,
in \cite{MR2075915} and mentioned in the proof of Proposition \ref{prop:crepant equal stringy}. \end{rem}
\begin{conjecture}
\label{conj: change of variables for untwisted }Let the assumption
be as in Section \ref{sec:The-untwisting-technique}. Let $\Phi:J_{\infty}X\supset A\to\cR\cup\{\infty\}$
be a measurable function with $A\subset p_{\infty}(J_{\infty}^{G,E}V)$
and let $p_{(\infty)}^{|F|}$ be the natural map $(J_{\infty}V^{|F|})/C_{G}(H)\to J_{\infty}X$.
We have 
\begin{align*}
\int_{A}\Phi\, d\mu_{J_{\infty}X} & =\int_{(p_{(\infty)}^{|F|})^{-1}(A)}(\Phi\circ p_{(\infty)}^{|F|})\LL^{-\bj_{p^{|F|}}}\, d\mu_{(J_{\infty}V^{|F|})/C_{G}(H)}.
\end{align*}

\end{conjecture}
This conjecture would be proved by using existing techniques and arguments
from \cite{MR1905024} and \cite{MR2075915}.
\begin{defn}[\cite{Yasuda:2013fk}.]
\label{def: weight}For $E\in\GCov(D)$ with a connected component
$F$, we define the \emph{weights }of $E$ and $F$ with respect to
$V$ as
\begin{align*}
\bw_{V}(E)=\bw_{V}(F) & :=\codim((V_{0})^{H},V_{0})-\bv_{V}(E)
\end{align*}
with
\[
\bv_{V}(E)=\bv_{V}(F):=\frac{1}{\sharp G}\cdot\length\frac{\Hom_{\cO_{D}}(M,\cO_{E})}{\cO_{E}\cdot\Xi_{F}}=\frac{1}{\sharp H}\cdot\length\frac{\Hom_{\cO_{D}}(M,\cO_{F})}{\cO_{F}\cdot\Xi_{F}}.
\]
For the generalization to the case where $k$ is only perfect, see
\cite{Wood-Yasuda-I}.
\end{defn}
The definition above gives the \emph{weight function,} 
\[
\bw_{V}:\GCov(D)\to\frac{1}{\sharp G}\ZZ.
\]
We will denote the composition 
\[
J_{\infty}^{G}V\to\GCov(D)\to\frac{1}{\sharp G}\ZZ
\]
again by $\bw_{V}$. 
\begin{defn}
For an ideal $I\subset\cO_{V}$ stable under the $G$-action and a
$G$-arc $\gamma:E\to V$, we define a function 
\[
\ord\, I:J_{\infty}^{G}V\to\frac{1}{\sharp G}\ZZ\cup\{\infty\}
\]
by 
\[
(\ord\, I)(\gamma):=\frac{1}{\sharp G}\length\frac{\cO_{E}}{\gamma^{-1}I}=\frac{1}{\sharp H}\length\frac{\cO_{F}}{(\gamma|_{F})^{-1}I}.
\]
We then extend this to $G$-stable fractional ideals and $G$-stable
$\QQ$-linear combinations of closed subschemes as in Definition \ref{def: order function}.
\end{defn}
The conjectural change of variables formula is stated as follows: 
\begin{conjecture}[\cite{Yasuda:2013fk}]
\label{conj:change-vars-linear}For a measurable function $\Phi:J_{\infty}X\supset C\to\cR\cup\{\infty\}$,
we have 
\[
\int_{C}\Phi\, d\mu_{J_{\infty}X}=\int_{p_{\infty}^{-1}(C)}(\Phi\circ p_{\infty})\LL^{-\bj_{p}+\bw_{V}}\, d\mu_{J_{\infty}^{G}V}.
\]

\end{conjecture}
To explain where the formula comes from, we first show a lemma:
\begin{lem}
\label{lem:explicit Jac}We have
\[
\Jac_{V^{\preuntwisting}/V\times_{D}F}=\fm_{F}^{\sharp H\cdot\bv_{V}(F)}\cO_{V^{\left\langle F\right\rangle }}.
\]
\end{lem}
\begin{proof}
Let $u':V^{\left\langle F\right\rangle }\to V\times_{D}F$ be the
natural map. We have the standard exact sequence
\[
(u')^{*}\Omega_{V\times_{D}F/F}\to\Omega_{V^{\preuntwisting}/F}\to\Omega_{V^{\preuntwisting}/V\times_{D}F}\to0.
\]
The left map is identical to the map 
\[
M\otimes_{\cO_{D}}\cO_{V^{\preuntwisting}}\to M^{\preuntwisting}\otimes_{\cO_{F}}\cO_{V^{\preuntwisting}}.
\]
Since the Fitting ideal is compatible with base changes (for instance,
see \cite[Cor. 20.5]{MR1322960}), if $I$ denotes the $0$th Fitting
ideal of 
\[
\mathrm{coker}\left(M\otimes_{\cO_{D}}\cO_{F}\to M^{\preuntwisting}\right),
\]
we have $\Jac_{V^{\preuntwisting}/V\times_{D}F}=I\cdot\cO_{V^{\preuntwisting}}$.
It is now easy to see that $I=\fm_{F}^{\sharp H\cdot\bv_{V}(F)}$,
for instance, by considering a triangular matrix representing the
map $M\otimes_{\cO_{D}}\cO_{F}\to M^{\preuntwisting}$ for suitable
bases. 
\end{proof}
Conjecture \ref{conj:change-vars-linear} can be guessed from the
following conjecture:
\begin{conjecture}
\label{conj: the key lemma}For $\gamma\in J_{\infty}^{G}V$ and $n\gg0$,
the fiber of the map
\[
p_{n}:\pi_{n}(J_{\infty}^{G}V)\to J_{n}X
\]
over the image of $\gamma$ is homeomorphic to a quotient of the affine
space 
\[
\AA_{k}^{(\bj_{p}-\bw_{V})(\gamma)}
\]
by a linear finite group action. 
\end{conjecture}
To see this, we first note that since two $G$-arcs $E\to V$ and
$E'\to V$ with $E\not\cong E'$ have distinct images in $J_{n}X$
for $n\gg0$, we can focus on $J_{\infty}^{G,E}V$ for fixed $E$.
Fixing a $G$-arc $\gamma:E\to V$, we consider the map
\[
(J_{n}V^{|F|})/C_{G}(H)\to J_{n}X.
\]
The fiber of this map over the image of $\gamma$ should be homeomorphic
to 
\[
\AA_{k}^{\bj_{p^{|F|}}(\gamma')}/A,
\]
where $\gamma'$ is an arc of $V^{|F|}$ corresponding to $\gamma$
and $A$ is a certain subgroup of $C_{G}(H)$ acting linearly on the
affine space. This fact would be proved in the course of proving Conjecture
\ref{conj: change of variables for untwisted }. On the other hand,
the map 
\[
\pi_{n}\left(\Hom_{F}^{H}(F,V^{\preuntwisting})\right)/C_{G}(H)\to\pi_{n}(J_{\infty}^{G}V)
\]
induced by $u$ has fibers homeomorphic to
\[
\AA_{k}^{\codim((V_{0})^{H},V_{0})}/B
\]
for some finite group $B$, which can be seen by looking at diagrams
(\ref{pentagon})-(\ref{pentagon-1-1}). From Lemma \ref{lem:explicit Jac},
\begin{align*}
 & \bj_{p^{|F|}}-\codim((V_{0})^{H},V_{0})\\
 & =\bj_{V^{\preuntwisting}/X\times_{D}F}-\codim((V_{0})^{H},V_{0})\\
 & =(\bj_{V\times_{D}F/X\times_{D}F}+\bj_{V^{\preuntwisting}/V\times_{D}F})-\codim((V_{0})^{H},V_{0})\\
 & =\bj_{p}-\bw_{V},
\end{align*}
concluding Conjecture \ref{conj: the key lemma}.

\section{The McKay correspondence for linear actions\label{sec:The-McKay-correspondence-linear}}

To state the McKay correspondence conjecture for linear actions, we
first define the notion of \emph{orbifold stringy motifs}. Keeping
the notation from the last section, let $\fX$, $\fV$, $\fV^{\left\langle F\right\rangle }$
and $\fV^{|F|}$ be centered log structures on $X$, $V$, $V^{\left\langle F\right\rangle }$
and $V^{|F|}$ respectively so that the following morphisms are all
crepant:
\[
\xymatrix{ & \fV^{\left\langle F\right\rangle }\ar[dr]\ar[dl]\\
\fV\ar[dr] &  & \fV^{|F|}\ar[dl]\\
 & \fX
}
\]
Since $X$ is $\QQ$-factorial, either $\fX$ or $\fV$ determines
the other centered log structures. The centered log structure $\fV$
is $G$-equivariant and $\fV^{|F|}$ $C_{G}(H)$-equivariant. 
\begin{defn}
We define the \emph{orbifold stringy motif }of the centered log $G$-$D$-variety
$\fV$ to be
\[
M_{\st}^{G}(\fV):=\int_{J_{\infty}^{G}\fV}\LL^{\bf_{\fV}+\bw_{V}}\, d\mu_{\fV}^{G}.
\]
Note that since $\fV$ is smooth over $D$, we have $\bf_{\fV}=\ord\,\Delta$
for the boundary $\Delta$ of $\fV$.
\end{defn}
Arguments as in the proof of Proposition \ref{prop:crepant equal stringy}
deduce the following conjecture from Conjecture \ref{conj:change-vars-linear}:
\begin{conjecture}[The motivic McKay correspondence for linear actions I]
\label{conj:wild mckay 1}We have
\[
M_{\st}(\fX)=M_{\st}^{G}(\fV).
\]

\end{conjecture}
We will next formulate a conjecture presented in a slightly different
way so that we will be able to generalize it to the non-linear case
easily. 
\begin{defn}
For $E\in\GCov(D)$, we define the \emph{$E$-parts }of $M_{\st}^{G}(\fV)$
and $M_{\st}(\fX)$ respectively by
\begin{align*}
M_{\st}^{G,E}(\fV): & =\int_{J_{\infty}^{G,E}\fV}\LL^{\bf_{\fV}+\bw_{V}}\, d\mu_{J_{\infty}^{G}\fV}\text{ and}\\
M_{\st}^{E}(\fX) & :=\int_{p_{\infty}(J_{\infty}^{G,E}\fV)}\LL^{\bf_{\fX}}\, d\mu_{J_{\infty}\fX}.
\end{align*}

\end{defn}
By the same reasoning as the one for the last conjecture, we would
have
\begin{equation}
M_{\st}^{G,E}(\fV)=M_{\st}^{E}(\fX).\label{eq:mckay E-part 1}
\end{equation}
On the other hand, from Conjecture \ref{conj: change of variables for untwisted },
we would have
\begin{equation}
M_{\st}^{E}(\fX)=M_{\st,C_{G}(H)}(\fV^{|F|}).\label{eq:mckay E-part 2}
\end{equation}

Let $\GCov(D)=\bigsqcup_{i=0}^{\infty}A_{i}$ be a conjectural stratification
with finite dimensional strata $A_{i}$ (see Remark \ref{rem: finite-dimensional strata}).
The author \cite{Yasuda:2013fk} conjectures also that each stratum
$A_{i}$ may not be of finite type over $k$, but the limit of a family
\[
X_{1}\xrightarrow{f_{1}}X_{2}\xrightarrow{f_{2}}\cdots
\]
such that $X_{j}$ are of finite type and $f_{i}$ are homeomorphisms.
We then define a \emph{constructible subset} of $\GCov(D)$ as a constructible
subset of $\bigsqcup_{i=0}^{n}A_{i}$ for some $n<\infty$, which
would be well-defined thanks to this conjecture. For a constructible
subset $C$ of $\GCov(D)$, its class $[C]$ in $\cR$ is well-defined.
Let $\tau$ denote the \emph{tautological motivic measure} on $\GCov(D)$
given by $\tau(C):=[C]$ for a constructible subset $C$. If a function
$\Phi:\GCov(D)\to\cR\cup\{\infty\}$ is constructible, that is, its
image is countable and all fibers $\Phi^{-1}(a)$, $a\in\cR$ are
constructible, then the integral $\int_{\GCov(D)}\Phi\, d\tau$ is
defined by
\[
\int_{\GCov(D)}\Phi\, d\tau:=\sum_{a\in\cR}\tau(\Phi^{-1}(a))\cdot a\in\cR\cup\{\infty\}.
\]
From Conjecture \ref{conj:wild mckay 1} and conjectural equations
(\ref{eq:mckay E-part 1}) and (\ref{eq:mckay E-part 2}), it seems
natural to expect
\[
M_{\st}^{G}(\fV)=\int_{\GCov(D)}M_{\st,C_{G}(H)}(\fV^{|F|})\, d\tau
\]
and hence: 
\begin{conjecture}[The motivic McKay correspondence for linear actions II]
\label{conj: wild McKay linear}We have
\[
M_{\st}(\fX)=\int_{\GCov(D)}M_{\st,C_{G}(H)}(\fV^{|F|})\, d\tau.
\]

\end{conjecture}
This formulation of the McKay correspondence is what we will generalize
to the non-linear case. 

To make this conjecture more meaningful, it would be nice if we can
compute $M_{\st,C_{G}(H)}(\fV^{|F|})$ explicitly. For this purpose,
next we see how to determine the centered log structures $\fV^{\left\langle F\right\rangle }$
and $\fV^{|F|}$ from $\fV$. Let us write $\fV=(V,\Delta,W)$, $\fV^{\preuntwisting}=(V,\Delta^{\preuntwisting},W^{\preuntwisting})$
and $\fV^{|F|}=(V,\Delta^{|F|},W^{|F|})$. The centers $W^{\preuntwisting}$
and $W^{|F|}$ are simply determined by 
\[
W^{\preuntwisting}=u^{-1}(W)\text{ and }W^{|F|}=r(W^{\preuntwisting}).
\]
The boundaries $\Delta^{\preuntwisting}$ and $\Delta^{|F|}$ are
determined as follows:
\begin{lem}
\label{lem:Delta transform}Regarding $V_{0}^{\left\langle F\right\rangle }$
and $V_{0}^{|F|}$ prime divisors on $V^{\left\langle F\right\rangle }$
and $V^{|F|}$, we have
\begin{gather*}
\Delta^{\preuntwisting}=u^{*}\Delta-(\sharp H\cdot\bv_{V}(E)+d_{F/D})\cdot V_{0}^{\preuntwisting}\\
\Delta^{|F|}=\frac{1}{\sharp H}\cdot r_{*}u^{*}\Delta-\bv_{V}(E)\cdot V_{0}^{|F|}.
\end{gather*}
Here $d_{F/D}$ is the different exponent of $F/D$, characterized
by $\Omega_{F/D}\cong\cO_{F}/\fm_{F}^{d_{F/D}}$. \end{lem}
\begin{proof}
For the first equality, we have 
\begin{align*}
 & u^{*}(K_{V}+\Delta)\\
 & =K_{V^{\preuntwisting}}-K_{V^{\preuntwisting}/V}+u^{*}\Delta\\
 & =K_{V^{\preuntwisting}}-K_{V^{\preuntwisting}/V\times_{D}F}-(u')^{*}K_{V\times_{D}F/V}+u^{*}\Delta.
\end{align*}
Here $K_{V^{\left\langle F\right\rangle }}$ is the canonical divisor
of $V^{\left\langle F\right\rangle }$ as a $D$-variety rather than
a $F$-variety and $u'$ denotes the natural morphism $V^{\preuntwisting}\to V\times_{D}F$.
From Lemma \ref{lem:explicit Jac}, 
\[
K_{V^{\preuntwisting}/V\times_{D}F}=\sharp H\cdot\bv_{V}(E)\cdot V_{0}^{\preuntwisting}.
\]
Since $(u')^{*}K_{V\times_{D}F/V}$ is the pull-back of $K_{F/D}$,
we have
\[
(u')^{*}K_{V\times_{D}F/V}=d_{F/D}\cdot V_{0}^{\preuntwisting}.
\]
These equalities show the first equality of the lemma. 

The second one follows from 
\begin{align*}
 & r^{*}(K_{V^{|F|}}+\frac{1}{\sharp H}\cdot r_{*}u^{*}\Delta-\bv(E)\cdot V_{0}^{|F|})\\
 & =K_{V^{\preuntwisting}}-K_{V^{\preuntwisting}/V^{|F|}}+u^{*}\Delta-\sharp H\cdot\bv_{V}(E)\cdot V_{0}^{\preuntwisting}\\
 & =K_{V^{\preuntwisting}}+u^{*}\Delta-(\sharp H\cdot\bv_{V}(E)+d_{F/E})\cdot V_{0}^{\preuntwisting}\\
 & =K_{V^{\left\langle F\right\rangle }}+\Delta^{\left\langle F\right\rangle }.
\end{align*}
\end{proof}
\begin{example}
\label{ex: classical wild McKay}Suppose that $\Delta=0$ and $W=\{o\}$
with $o\in V_{0}$ the origin. Then $\Delta^{|F|}=-\bv_{V}(E)\cdot V_{0}^{|F|}$
and $W^{|F|}\cong\AA_{k}^{\codim((V_{0})^{H},V_{0})}$. Hence
\[
M_{\st}^{G,E}(\fV)=M_{\st,C_{G}(H)}(\fV^{|F|})=\LL^{\bw_{V}(E)}.
\]
Conjecture \ref{conj: wild McKay linear} is reduced to the form,
\begin{equation}
M_{\st}(\fX)=\int_{\GCov(D)}\LL^{\bw_{V}}\, d\tau.\label{eq:wild McKay old form}
\end{equation}
If $p:V\to X$ is \'{e}tale in codimension one and if we denote $p(o)$
again by $o$, then $M_{\st}(\fX)=M_{\st}(X)_{o}$ and the last equality
is exactly what was conjectured in \cite{Yasuda:2013fk}. \end{example}
\begin{rem}
If $\sharp G$ is prime to the characteristic of $k$, then $\GCov(D)$
is identified with the set of conjugacy classes of $G$, denoted by
$\Conj G$. Equality (\ref{eq:wild McKay old form}) in the last example
is then written as
\[
M_{\st}(\fX)=\sum_{[g]\in\Conj G}\LL^{\bw_{V}(g)}.
\]
Expressing the weights $\bw_{V}(g)$ in terms of eigenvalues, we recover
results by Batyrev \cite{MR1677693}, and Denef and Loeser \cite{MR1905024}.
\end{rem}

\section{The McKay correspondence for non-linear actions\label{sec:The-McKay-correspondence-non-linear}}

In this section, we generalize Conjecture \ref{conj: wild McKay linear}
to the non-linear case. It is rather easy, once we have formulated
the conjecture as it is. 

Let us consider an affine $D$-variety $\sv=\Spec\cO_{\sv}$ endowed
with a faithful $G$-action. We fix a $G$-equivarint (locally closed)
immersion 
\[
\sv\hookrightarrow V
\]
into an affine space $V\cong\AA_{D}^{d}$ endowed with a linear $G$-action.
Identifying $G$-arcs of $\sv$ with those of $V$ factoring through
$\sv$, we regard $J_{\infty}^{G}\sv$ as a subset of $J_{\infty}^{G}V$.
\begin{rem}
\label{rem: linear embeding (permutation)}Such an immersion always
exists. Indeed, let $f_{1},\dots,f_{n}$ be generators of $\cO_{\sv}$
as an $\cO_{D}$-algebra, let $A:=\bigcup_{i}f_{i}G$, the union of
their orbits, and let $\cO_{D}[x_{f}\mid f\in A]$ be the polynomial
ring with variables $x_{f}$, $f\in A$ over $\cO_{D}$. The ring
has a natural $G$-action by permutations of variables. The $\cO_{D}$-algebra
homomorphism 
\[
\cO_{D}[x_{f}\mid f\in A]\to\cO_{\sv},\, x_{f}\mapsto f
\]
defines a desired immersion. Moreover this construction gives a \emph{closed}
immersion into $V$ on which $G$ acts by \emph{permutations}. In
this case, our weigh function $\bw_{V}$ is closely related to the
Artin and Swan conductors \cite{Wood-Yasuda-I}, although we do not
use this fact in this paper.\end{rem}
\begin{defn}
For $E\in\GCov(D)$ with a connected component $F$, we define the
\emph{pre-untwisting variety }of $\sv$, denoted by $\sv^{\preuntwisting}$,
as the irreducible component of $r^{-1}(\sv)\subset V^{\preuntwisting}$
which dominates $\sv$. We then define the \emph{untwisting variety},
denoted by $\sv^{|F|}$, as the image of $\sv^{\left\langle F\right\rangle }$
in $V^{|F|}$. We also define the \emph{normalized pre-untwisting}
$\sv^{\preuntwisting,\nu}$ \emph{and untwisting varieties} $\sv^{|F|,\nu}$
to be the normalizations of $\sv^{\preuntwisting}$ and $\sv^{|F|}$
respectively. 
\end{defn}
Let $\sx:=\sv/G$. The following diagram shows natural morphisms of
relevant varieties and symbols $t$, $s$ and $q$ denote morphisms
as indicated:

\begin{equation}
\xymatrix{ & \sv^{\preuntwisting,\nu}\ar[d]\ar[dr]^{s}\ar[ddl]_{t}\\
 & \sv^{\preuntwisting}\ar[dr]\ar[dl] & \sv^{|F|,\nu}\ar[d]\\
\sv\ar[dr]_{q} &  & \sv^{|F|}\ar[dl]\\
 & \sx
}
\label{eq: diagram s and t}
\end{equation}
The one-to-one correspondence obtained in the last section
\[
J_{\infty}^{G,E}V\leftrightarrow(J_{\infty}V^{|F|})/C_{G}(H)
\]
induces a one-to-one correspondence
\[
J_{\infty}^{G,E}\sv\leftrightarrow(J_{\infty}\sv^{|F|})/C_{G}(H).
\]
We obtain the following diagram:
\[
\xymatrix{ &  & (J_{\infty}\sv^{|F|,\nu})/C_{G}(H)\ar[d]\\
J_{\infty}^{G,E}\sv\ar@{<->}[rr]^{\text{1-to-1}}\ar[dr] &  & (J_{\infty}\sv^{|F|})/C_{G}(H)\ar[dl]\\
 & J_{\infty}\sx
}
\]
If we put $J_{\infty}^{E}\sx$ to be the image of $J_{\infty}^{G,E}\sv$
in $J_{\infty}\sx$, then we can naturally expect that $J_{\infty}^{E}\sx$
coincides with the images of $J_{\infty}\sv^{|F|}$ and $J_{\infty}\sv^{|F|,\nu}$
modulo measure zero subsets.  

From now on, we suppose that $\sv$ is normal. Let $\fv$, $\fv^{\left\langle F\right\rangle ,\nu}$,
$\fv^{|F|,\nu}$ and $\fx$ be centered log structures on $\sv$,
$\sv^{\left\langle F\right\rangle ,\nu}$ and $\sv^{|F|,\nu}$ respectively
such that the morphisms 
\[
\xymatrix{ & \fv^{\left\langle F\right\rangle ,\nu}\ar[dr]\ar[dl]\\
\fv\ar[dr] &  & \fv^{|F|,\nu}\ar[dl]\\
 & \fx
}
\]
are all crepant. The centered log $D$-varieties $\fv$ and $\fv^{|F|,\nu}$
are $G$- and $C_{G}(H)$-equivariant respectively. If we define the
$E$-part $M_{\st}^{E}(\fx)$ of $M_{\st}(\fx)$, we can expect 
\[
M_{\st}^{E}(\fx)=M_{\st,C_{G}(H)}(\fv^{|F|,\nu})
\]
similarly to the linear case. For the equality is a slight generalization
of Conjecture \ref{conj: crepant equal stringy} and would follow
from the change of variables formula generalized along the line of
\cite{MR1905024}, applied to the almost bijection
\[
J_{\infty}\fv^{|F|,\nu}\to J_{\infty}^{E}\fx.
\]
It is then natural to expect:
\begin{conjecture}[The McKay correspondence for non-linear actions]
\label{conj: non-linear McKay-1}We have 
\[
M_{\st}(\fx)=\int_{\GCov(D)}M_{\st,C_{G}(H)}(\fv^{|F|,\nu})\, d\tau.
\]
\end{conjecture}
\begin{defn}
We define the $E$\emph{-part of the orbifold stringy motif }of $\fv$
as
\begin{align*}
M_{\st}^{G,E}(\fv) & :=M_{\st,C_{G}(H)}(\fv^{|F|,\nu})
\end{align*}
and the \emph{orbifold stringy motif }of $\fv$ as

\[
M_{\st}^{G}(\fv):=\int_{\GCov(D)}M_{\st}^{G,E}(\fv)\, d\tau.
\]

\end{defn}
With this definition, the last conjecture simply says
\[
M_{\st}(\fx)=M_{\st}^{G}(\fv).
\]
 
\begin{rem}
The reader may wonder why we do not define $M_{\st}^{G}(\fv)$ as
a motivic integral on $J_{\infty}^{G}\fv$, which appears more natural.
It is because the author does not know whether one can define a motivic
measure on $J_{\infty}^{G}\fv$. For, he does not know how to compute
dimensions of fibers of
\[
\pi_{n}(\Hom_{F}^{H}(F,\sv^{\left\langle F\right\rangle ,\nu}))/C_{G}(H)\to\pi_{n}(J_{\infty}^{G}\sv).
\]
Knowing it was, in the linear case, a key in formulating the change
of variables formula (Conjecture \ref{conj:change-vars-linear}) and
determining the integrand $\LL^{\bf_{\fV}+\bw_{V}}$ in the definition
of $M_{\st}^{G}(\fV)$. 
\end{rem}

\section{Computing boundaries of untwisting varieties\label{sec:Computing-boundaries}}

To compute examples of the wild McKay correspondence, we need to determine
centered log varieties $\fv^{|F|,\nu}$. It is easy to determine the
center. In this section, supposing $\sv$ and $\sv^{|F|}$ are both
normal and complete intersections in $V$ and $V^{|F|}$ respectively,
we compute the boundary of $\fv^{|F|}$.

Let
\[
t:\sv^{\left\langle F\right\rangle }\to\sv\text{ and }s:\sv^{\left\langle F\right\rangle }\to\sv^{|F|}
\]
be the natural morphisms, although they are different from the morphisms
denoted by the same symbols in diagram (\ref{eq: diagram s and t})
unless $\sv^{\left\langle F\right\rangle }$ is also normal. The subvariety
$\sv^{\left\langle F\right\rangle }\subset V^{\left\langle F\right\rangle }$
is a complete intersection. To see this, first note that if $s^{-1}(\sv^{|F|})$
denotes the scheme-theoretic preimage, then $\left(s^{-1}(\sv^{|F|})\right)_{\red}=\sv^{\left\langle F\right\rangle }$.
The subscheme $s^{-1}(\sv^{|F|})\subset V^{\left\langle F\right\rangle }$
is a complete intersection, hence Cohen-Macaulay, and generically
reduced. From \cite[THeorem 18.15]{MR1322960}, $s^{-1}(\sv^{|F|})$
is actually reduced and 
\[
s^{-1}(\sv^{|F|})=\sv^{\left\langle F\right\rangle }.
\]

In general, for a complete intersection subvariety $Y\subset X$,
its \emph{conormal sheaf} $\cC_{Y/X}$ is defined as $I_{Y}/I_{Y}^{2}$
with $I_{Y}\subset\cO_{X}$ the defining ideal sheaf of $Y$. We put
\[
\det\cC_{Y/X}:=\bigwedge^{\codim(Y,X)}\cC_{Y/X}.
\]
There exists a unique effective $H$-stable Cartier divisor $A_{F}$
on $\sv^{\left\langle F\right\rangle }$ such that 
\[
t^{*}\left(\det\cC_{\sv/V}\right)=\left(\det\cC_{\sv^{\left\langle F\right\rangle }/V^{\left\langle F\right\rangle }}\right)(-A_{F}).
\]

\begin{prop}
\label{prop: complete intersection}Let $\delta$ and $\delta^{|F|}$
be the boundaries of $\fv$ and $\fv^{|F|}$ respectively and $C:=V_{0}^{|F|}|_{\sv^{|F|}}$,
the restriction of the prime divisor $V_{0}^{|F|}$ on $V^{|F|}$
to $\sv^{|F|}$. Then 
\[
\delta^{|F|}=\frac{1}{\sharp H}\cdot s_{*}\left(t^{*}\delta+A_{F}\right)-\bv_{V}(E)\cdot C.
\]
\end{prop}
\begin{proof}
Let $\epsilon^{|F|}$ be the right side of the equality. As in the
proof of Lemma \ref{lem:Delta transform}, it suffices to show that
the pull-backs of divisors $K_{\sv}+\delta$ and $K_{\sv^{|F|}}+\epsilon^{|F|}$
to $\sv^{\left\langle F\right\rangle }$ coincide. Since 
\[
s^{*}(\frac{1}{\sharp H}s_{*}t^{*}\delta)=t^{*}\delta,
\]
we may suppose $\delta=0$ and hence 
\[
\epsilon^{|F|}=\frac{1}{\sharp H}\cdot s_{*}A_{F}-\bv_{V}(E)C.
\]
By abuse of notation, identifying a divisor corresponding to an invertible
sheaf, from the adjunction formula, we have 
\begin{align*}
t^{*}K_{\sv} & =t^{*}\left(K_{V}|_{\sv}-\det\cC_{\sv/V}\right)\\
 & =\left(u^{*}K_{V}\right)|_{\sv^{\left\langle F\right\rangle }}-\det\cC_{\sv^{\left\langle F\right\rangle }/V^{\left\langle F\right\rangle }}+A_{F}.
\end{align*}
On the other hand, since $s^{*}(\det\cC_{\sv^{|F|}/V^{|F|}})=\det\cC_{\sv^{\left\langle F\right\rangle }/V^{\left\langle F\right\rangle }}$,
\begin{align*}
s^{*}(K_{\sv^{|F|}}+\epsilon^{|F|}) & =s^{*}(K_{V^{|F|}}|_{\sv^{|F|}}-\det\cC_{\sv^{|F|}/V^{|F|}})-\sharp H\cdot\bv_{V}(E)V_{0}^{\left\langle F\right\rangle }|_{\sv^{\left\langle F\right\rangle }}+A_{F}\\
 & =\left(r^{*}K_{V^{|F|}}-\sharp H\cdot\bv_{V}(E)V_{0}^{\left\langle F\right\rangle }\right)|_{\sv^{\left\langle F\right\rangle }}-\det\cC_{\sv^{\left\langle F\right\rangle }/V^{\left\langle F\right\rangle }}+A_{F}.
\end{align*}
From Lemma \ref{lem:Delta transform}, 
\[
r^{*}K_{V^{|F|}}-\sharp H\cdot\bv_{V}(E)V_{0}^{\left\langle F\right\rangle }=K_{V^{\left\langle F\right\rangle }/D}-d_{F/D}V_{0}^{\left\langle F\right\rangle }-\sharp H\cdot\bv_{V}(E)V_{0}^{\left\langle F\right\rangle }=u^{*}K_{V},
\]
which shows the proposition. 
\end{proof}
It is handy to rewrite the proposition in the case of hypersurfaces
as follows.
\begin{cor}
\label{cor: hypersurface by an invariant polynomial}Suppose that
$\sv\subset V$ is a hypersurface defined by a polynomial $f\in\cO_{V}$
and write 
\[
u_{F}^{*}f=\pi_{F}^{b}\phi,
\]
where $\pi_{F}$ is a uniformizer of $\cO_{F}$, $b$ is an integer
$b\ge0$ and $\phi\in\cO_{V^{\left\langle F\right\rangle }}$ with
$\pi_{F}\nmid\phi$. Then, with the notation as above, we have
\[
\delta^{|F|}=\frac{1}{\sharp H}s_{*}t^{*}\delta+\left(\frac{b}{\sharp H}-\bv_{V}(E)\right)\cdot C.
\]
\end{cor}
\begin{proof}
The corollary follows from 
\[
A_{F}=b\cdot(V_{0}^{\left\langle F\right\rangle }|_{\sv_{0}^{\left\langle F\right\rangle }})\text{ and }s_{*}A_{F}=b\cdot C.
\]
\end{proof}
\begin{rem}
In the corollary above, if $f$ is $G$-invariant, then $b$ is a
multiple of $\sharp H$ and hence $b/\sharp H$ is an integer. 
\end{rem}

\section{A tame singular example\label{sec:A-tame-example}}

In this section, we verify Conjecture \ref{conj: non-linear McKay-1}
in an example from the tame case, where $\sv$ is not regular. 

Suppose that $k$ has characteristic $\ne2$. Let $D:=\Spec k[[\pi]]$,
$V:=\Spec k[[\pi]][x,y,z]$ and $\sv=\Spec k[[\pi]][x,y,z]/(xz-y^{2})$,
the trivial family of the $A_{1}$-singularity over $\Spec k[[\pi]]$.
We suppose that $G=\ZZ/2\ZZ=\{1,g\}$ acts on $V$ by
\[
xg=-x,\, yg=y,\, zg=-z.
\]
The subvariety $\sv$ is stable under the $G$-action and the quotient
variety $\sx=\sv/G$ can be embedded into $\AA_{k[[\pi]]}^{3}=\Spec k[[\pi]][u,v,w]$
and gives the hypersurface defined by the equation $uv-w^{4}=0$.
Thus $\sx$ is the trivial family of the $A_{3}$-singularity over
$\Spec k[[\pi]]$. 

Since the morphism $\sv\to\sx$ is \'{e}tale in codimension one,
it is crepant (with the identification (\ref{eq:identifications center log})).
Let $\tilde{\sx}_{0}\to\sx_{0}$ be the minimal resolution and $\tilde{\sx}:=\sx\otimes_{k}k[[\pi]]$.
The natural morphism $\tilde{\sx}\to\sx$ is crepant. From Proposition
\ref{prop: ordinary explicit formula}, 
\[
M_{\st}(\sx)=M_{\st}(\tilde{\sx})=[\tilde{\sx}_{0}]=\LL^{2}+3\LL.
\]

Next we will compute $M_{\st}^{G}(\sv)$ and verify that it coincides
with $M_{\st}(\sx)$. There are exactly two $G$-covers of $D$ up
to isomorphism: the trivial one $E_{1}=D\sqcup D\to D$ and the nontrivial
one 
\[
E_{2}=\Spec k[[\pi^{1/2}]]\to D=\Spec k[[\pi]],
\]
and hence
\[
M_{\st}^{G}(\sv)=M_{\st}^{G,E_{1}}(\sv)+M_{\st}^{G,E_{2}}(\sv).
\]
As for the first term $M_{\st}^{G,E_{1}}(\sv)$, we have $\sv^{|D|}=\sv$.
Consider the minimal resolution $\tilde{\sv}_{0}\to\sv_{0}$ and put
$\tilde{\sv}:=\tilde{\sv}_{0}\otimes_{k}k[[\pi]]$. Then the morphism
$\tilde{\sv}\to\sv$ is crepant. Since the $G$-action on the exceptional
locus is trivial, from Proposition \ref{prop:equivariant explicit formula},
\[
M_{\st}^{G,E_{1}}(\sv)=M_{\st,G}(\tilde{\sv})=\LL^{2}+\LL.
\]
Next we compute $M_{\st}^{G,E_{2}}(\sv)$. For $F=E_{2}$, the tuning
module $\Xi_{F}$ has a basis
\begin{equation}
\pi^{1/2}x^{*},y^{*},\pi^{1/2}z^{*},\label{eq:basis xi tame example}
\end{equation}
with $x^{*},y^{*},z^{*}$ the dual basis of $ $$x$, $y$, $z$.
If we denote the dual basis of (\ref{eq:basis xi tame example}) by
$\rx$, $\ry$, $\rz$, then we can write $u^{*}$ as
\begin{align*}
u^{*}:k[[\pi]][x,y,z] & \to k[[\pi^{1/2}]][\rx,\ry,\rz]\\
x & \mapsto\pi^{1/2}\rx\\
y & \mapsto\ry\\
z & \mapsto\pi^{1/2}\rz.
\end{align*}
We see that $\sv^{|F|}$ is given by
\[
\pi\rx\rz-\ry^{2}=0.
\]
Since the non-regular locus of $\sv^{|F|}$ has dimension one, the
variety $\sv^{|F|}$ is normal. From Corollary \ref{cor: hypersurface by an invariant polynomial},
the boundary $\delta^{|F|}$ of $\fv^{|F|}$ is given by
\[
\delta^{|F|}=-V_{0}^{|F|}|_{\sv^{|F|}}.
\]
Hence 
\[
M_{\st}^{G,E_{2}}(\sv)=M_{\st,G}(\fv^{|F|})=M_{\st,G}(\sv^{|F|})\LL^{-1}.
\]
The $G$-action on $\sv^{|F|}$ is given by 
\[
\rx g=-\rx,\,\ry g=\ry,\,\rz g=-\rz.
\]
The non-regular locus of $\sv^{|F|}$ consists of three irreducible
components
\[
C_{1}=\{\rx=\ry=\rz=0\},\, C_{2}=\{\rx=\ry=\pi=0\},\, C_{3}=\{\ry=\rz=\pi=0\}.
\]
Let $\mathsf{\sv_{1}}\to\sv^{|F|}$ be the blowup along $C_{1}$.
Then the non-regular locus of $\mathsf{\sv}_{1}$ is exactly the union
of the strict transforms $C_{2}'$ and $C_{3}'$ of $C_{2}$ and $C_{3}$.
Moreover the singularities of $\mathsf{\sv}_{1}$ are two trivial
families of the $A_{1}$-singularity over $\AA_{k}^{1}$. Let $\mathsf{\sv}_{2}\to\mathsf{\sv_{1}}$
be the blowup along $C_{2}'$ and $C_{3}'$. Then $\mathsf{\sv}_{2}$
is regular. If $A_{2}$ and $A_{3}$ are the exceptional prime divisors
over $C_{2}'$ and $C_{3}'$ respectively, then the smooth locus of
$\mathsf{\sv}_{2}\to D$ in the special fiber is the disjoint union
of open subsets $A_{2}'\subset E_{2}$ and $A_{3}'\subset E_{3}$
with $A_{2}'\cong A_{3}'\cong\AA_{k}^{2}$. Since the morphism $\mathsf{\sv}_{2}\to\sv^{|F|}$
is crepant and the $G$-action on its exceptional locus is trivial,
\[
M_{\st,G}(\sv^{|F|})=M_{\st,G}(\mathsf{\sv}_{2})=[A_{2}'\sqcup A_{3}']=2\LL^{2}
\]
and 

\[
M_{\st}^{G}(\sv)=M_{\st}^{G,E_{1}}(\sv)+M_{\st}^{G,E_{2}}(\sv)=\LL^{2}+3\LL,
\]
as desired.

\section{A wild non-linear example\label{sec:A-wild-example}}

In this section, we compute an example from the wild case. 

Suppose that $k$ has characteristic two. Let $V:=\Spec k[[\pi]][x,y]$
on which the group $G=\{1,g\}\cong\ZZ/2\ZZ$ acts by the transposition
of $x$ and $y$, and $\sv:=\Spec k[[\pi]][x,y]/(x+y+xy)$. The completion
of $\sv$ at the origin $o\in\sv_{0}\subset V_{0}$ gives 
\[
\Spec k[[\pi,x]]
\]
with the $G$-action by
\[
xg=\frac{x}{1+x}=x+x^{2}+x^{3}+\cdots.
\]
The invariant subring of $k[[\pi,x]]$ is
\[
k[[\pi,x]]^{g}=k[[\pi,\frac{x^{2}}{1+x}]].
\]
Since 
\[
k[[x]]=\frac{k[[\frac{x^{2}}{1+x}]][X]}{\left\langle F(X)\right\rangle },\, F(X):=X^{2}+\frac{x^{2}}{1+x}X+\frac{x^{2}}{1+x},
\]
the different of $k[[x]]/k[[\frac{x^{2}}{1+x}]]$ is 
\[
\left\langle F'(x)\right\rangle =\left\langle x^{2}\right\rangle 
\]
(see \cite[page 56, Cor. 2]{MR554237}). Let $Z\subset V$ be the
zero section of $V\to D$, defined by the ideal $\left\langle x,y\right\rangle \subset k[[\pi]][x,y]$.
We regard $Z$ as a prime divisor on $\sv$. Note that $2Z$ is defined
by $x+y=0$. 

If we put $\fv=(\sv,\delta=-2Z,o)$ and $\fx=(\sx,0,\bar{o})$ with
$\bar{o}$ the image of $o$, then the quotient morphism $q:\fv\to\fx$
is crepant. We obviously have 
\[
M_{\st}(\fx)=1.
\]
 Next we will verify that $M_{\st}^{G}(\fv)=1$. For the trivial $G$-cover
$E_{1}=D\sqcup D\to D$, since $\sv^{|D|}=\sv$, from Proposition
\ref{prop:equivariant explicit formula}, we have
\[
M_{\st}^{G,E_{1}}(\fv)=\frac{\LL-1}{\LL^{3}-1}=\frac{1}{\LL^{2}+\LL+1}.
\]
Let $E=F=\Spec k[[\rho]]$ be any non-trivial $G$-cover of $D=\Spec k[[\pi]]$.
The associated tuning module $\Xi_{F}$ is generated by two elements
$\alpha_{1}$ and $\alpha_{2}$ given by 
\[
\alpha_{1}:x\mapsto1,\, y\mapsto1
\]
and 
\[
\alpha_{2}:x\mapsto\rho,\, y\mapsto\rho g.
\]
Let $\rx$ and $\ry$ be the dual basis of $\alpha_{1}$ and $\alpha_{2}$.
Then $u^{*}$ is given by
\begin{align*}
k[[\pi]][x,y] & \to k[[\rho]][\rx,\ry]\\
x & \mapsto\rx+\rho\ry\\
y & \mapsto\rx+(\rho g)\ry.
\end{align*}
Therefore $\sv^{\left\langle F\right\rangle }$ and $\sv^{|F|}$ are
defined by
\begin{align*}
 & (\rx+\rho\ry)(\rx+(\rho g)\ry)+(\rx+\rho\ry)+(\rx+(\rho g)\ry)\\
 & =\rx^{2}+\Nr(\rho)\ry^{2}+\Tr(\rho)\ry(1+\rx)\\
 & =0.
\end{align*}
The $G$-action on $k[[\rho]][\rx,\ry]$ is given by
\begin{equation}
\rx g=\rx,\,\ry g=\frac{\rho g}{\rho}\ry.\label{eq:wild action example}
\end{equation}
The pull-back of $2Z$ to $\sv^{\left\langle F\right\rangle }$ is
defined by $\Tr(\rho)\ry$. Let $S:=\sv_{0}^{|F|}$, regarded as a
prime divisor on $\sv^{|F|}$ and let $B$ be the prime divisor on
$\sv^{|F|}$ such that $2B$ is defined by $\ry=0$. From Corollary
\ref{cor: hypersurface by an invariant polynomial}, the boundary
$\delta^{|F|}$ of $\fv^{|F|}$ is 
\[
-4nS-2B
\]
with $n\in\ZZ_{>0}$ given by $\left\langle \Tr(\rho)\right\rangle =\left\langle \pi^{n}\right\rangle $.
The center of $\fv^{|F|}$ is $\sv_{0}^{|F|}$. Hence
\[
M_{\st,G}(\fv^{|F|})=M_{\st,G}(\sv^{|F|},-2B)\LL^{-2n}.
\]

Let us now consider the case $n=1$. The variety $\sv^{|F|}$ has
two $A_{1}$-singularities at 
\[
(\rx,\ry,\pi)=(0,0,0),\,(0,1,0).
\]
Blowing them up, we get a crepant morphism $\tilde{\sv}^{|F|}\to\sv^{|F|}$.
Let $N_{0}$ and $N_{1}$ be the exceptional prime divisors over $(0,0,0)$
and $(0,1,0)$ respectively. The $G$-action on $N_{0}$ is trivial
and the one on $N_{1}$ linear. Let $\tilde{B}\subset\tilde{\sv}^{|F|}$
be the strict transform of $B$. The morphism $(\tilde{\sv}^{|F|},-2\tilde{B}-N_{0})\to(\sv^{|F|},-2B)$
is crepant. Since $\tilde{\sv}^{|F|}$ is regular and the smooth locus
of $\tilde{\sv}^{|F|}\to D$ in the special fiber is 
\[
N_{0}\setminus\{1\text{ point}\}\sqcup N_{1}\setminus\{1\text{ point}\},
\]
where the removed point of $N_{0}$ is different from the intersection
$N_{0}\cap\tilde{B}$. Therefore 
\begin{align*}
M_{\st,G}(\sv^{|F|},-2B) & =M_{\st,G}(\tilde{\sv}^{|F|},-2\tilde{B}-N_{0})\\
 & =\LL+\left((\LL-1)+\frac{\LL-1}{\LL^{3}-1}\right)\LL^{-1}\\
 & =\frac{\LL(\LL+1)^{2}}{\LL^{2}+\LL+1}.
\end{align*}

Next consider the case $n\ge2$. Then $\sv^{|F|}$ is non-regular
only at the origin $o=(0,0,0)$. The completion of $\sv^{|F|}$ at
the origin is 
\[
\Spec\frac{k[[\pi,\rx,\ry]]}{\left\langle \rx^{2}+\pi\ry^{2}+\pi^{n}\ry\right\rangle }
\]
after a suitable change of coordinates, which is the $D_{2n}^{0}$-singularity
in Artin's classification \cite{MR0450263}. Let $f:\tilde{\sv}^{|F|}\to\sv^{|F|}$
be the minimal resolution. The exceptional prime divisors $N_{1},\dots,N_{2n}$
and the strict transform $\tilde{B}$ of $B$ and the one $\tilde{S}$
of $S$ are arranged as indicated in the following dual graph:
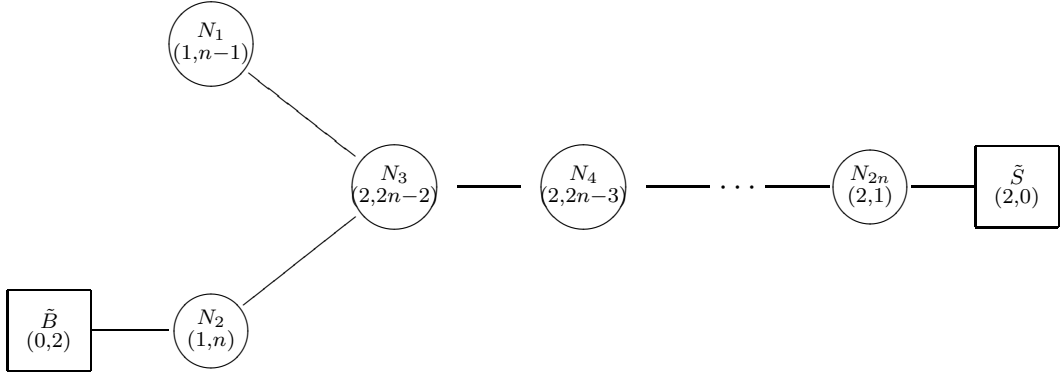
\begin{equation*}
\xymatrix{
&*++[o][F]{ \substack{N_1 \\ (1,n-1)}} \ar@{-}[dr]\\
&&*++[o][F]{ \substack{N_3 \\ (2,2n-2)}}\ar@{-}[r] &*++[o][F]{ \substack{N_4 \\ (2,2n-3)}}\ar@{-}[r] & \cdots\ar@{-}[r] &*++[o][F]{ \substack{N_{2n} \\ (2,1)}}\ar@{-}[r]& *++[F]{ \substack{\tilde{S} \\ (2,0)}}\\
*++[F]{ \substack{\tilde{B} \\ (0,2)}}\ar@{-}[r]&*++[o][F]{ \substack{N_2 \\ (1,n)}}\ar@{-}[ur]
}
\end{equation*}Here the pairs of numbers, say $(a,b)$, mean that $a$ is the multiplicity
of the relevant prime divisor in $f^{*}(2S)$ and $b$ the one in
$f^{*}(2B)$. If we put 
\[
\tilde{\delta}^{|F|}:=-2\tilde{B}-(n-1)N_{1}-nN_{2}-\sum_{i=2}^{2n-1}(2n-i)N_{i+1},
\]
then the morphism 
\[
(\tilde{\sv}^{|F|},\tilde{\delta}^{|F|})\to(\sv^{|F|},-2B)
\]
is crepant. Since $N_{1}$ and $N_{2}$ are the only prime divisors
having multiplicity one in $f^{*}(2S)$, the smooth locus of the morphism
$\tilde{\sv}^{|F|}\to D$ in the special fiber is 
\[
(N_{1}\sqcup N_{2})\setminus N_{3}.
\]
Since the $G$-action on the exceptional locus of $f$ is trivial,
we have
\begin{align*}
M_{\st,G}(\tilde{\sv}^{|F|},\tilde{\delta}^{|F|}) & =\LL\cdot\LL^{-n+1}+\left((\LL-1)+\frac{\LL-1}{\LL^{3}-1}\right)\LL^{-n}\\
 & =\frac{(\LL+1)^{2}\LL^{2-n}}{\LL^{2}+\LL+1}.
\end{align*}

In summary, for $n>0$, we have
\[
M_{\st}^{G,E}(\fv)=\frac{(\LL+1)^{2}\LL^{2-3n}}{\LL^{2}+\LL+1}
\]
Since the locus of $E\in\GCov(D)$ with $\ord_{\pi}\,\Tr(\rho)=n$
is homeomorphic to $\GG_{m,k}\times\AA_{k}^{n-1}$ (see \cite{Yasuda:2012fk}),
\begin{align*}
M_{\st}^{G}(\fv) & =M_{\st}^{G,E_{1}}(\fv)+\int_{\GCov(D)\setminus\{E_{1}\}}M_{\st}^{G,E}(\fv)\, d\tau\\
 & =\frac{1}{\LL^{2}+\LL+1}+\sum_{n=1}^{\infty}\frac{(\LL+1)^{2}\LL^{2-3n}}{\LL^{2}+\LL+1}\times(\LL-1)\LL^{n-1}\\
 & =1.
\end{align*}

\section{Stable hyperplanes in permutation representations\label{sec:stable hyperplanes}}

When $G$ acts on $V$ by permutations of coordinates, then functions
$\bv_{V}$ and $\bw_{V}$ can be computed by using Artin or Swan conductors,
or discrminants or differents \cite{Wood-Yasuda-I}. In this section,
we generalize it to the case of a hyperplane in a permutation representation
defined by an invariant linear form.

Suppose that $G$ acts on 
\[
V=\AA_{D}^{d}=\Spec\cO_{D}[x_{1},\dots,x_{d}]
\]
by permutations of coordinates and 
\[
\AA_{D}^{d-1}\cong\sv=\Spec\cO_{V}/\left\langle f\right\rangle \subset V
\]
is a hyperplane defined by a $G$-invariant linear form 
\[
f=\sum_{i=1}^{d}f_{i}x_{i}\in M^{G}\quad\left(M:=\bigoplus_{i=1}^{d}\cO_{D}x_{i}\right).
\]
The assumption that $\AA_{D}^{d-1}\cong\sv$ means that at least one
coefficient $f_{i}$ is a unit in $\cO_{D}$. 

Fix $E\in\GCov(D)$ and a connected component $F$ of $E$ with stabilizer
$H$. Let 
\[
\{x_{1},x_{2},\dots,x_{d}\}=O_{1}\sqcup O_{2}\sqcup\cdots\sqcup O_{l}
\]
be the decomposition into the $H$-orbits. Reordering $x_{1},\dots,x_{d}$
if necessary, we suppose that 
\[
O_{j}=x_{j}H,\,1\le j\le l.
\]
The assumption $f\in M^{G}$ now means that if $i\in O_{j}$ and if
$h_{i}\in H$ is any element sending $x_{j}$ to $x_{i}$, then 
\[
f_{i}=f_{j}h_{i}.
\]
For $1\le j\le l$, we put $H_{j}\subset H$ to be the stabilizer
of $j$, which has order $\sharp H/\sharp O_{j}$, and put $C:=\Spec\left(\cO_{F}\right)^{H_{j}}$,
which is a  cover of $D$ of degree $\sharp O_{j}$. Accordingly 
\[
C:=\bigsqcup_{j=1}^{l}C_{j}\to D
\]
is a cover of degree $d$. Here we say that a morphism $C\to D$ is
a cover if $C$ is the normalization of $D$ in some finite \'{e}tale
(not necessarily Galois) $K(D)$-algebra. We obtain $C$ from $E$
also in the following way. If $G_{D}$ is the absolute Galois group
of $K(D)$, then the $G$-cover $E$ corresponds to a continuous homomorphism
$\rho:G_{D}\to G$ (up to conjugation). Since $G$ acts on $\{1,\dots,d\}$
by conjugation, we get a continuous action of $G_{D}$ on $\{1,\dots,d\}$,
giving a finite \'{e}tale cover $C^{\circ}\to\Spec K(D)$. Taking
the normalization of $D$ in $C^{\circ}$, we get $C$ (up to isomorphism). 

For a cover $C\to D$, we denote by $d_{C/D}$ its \emph{discriminant
exponent}: the discriminant of the extension $K(C)/K(D)$ is $\fm_{D}^{d_{C/D}}$.
If $C$ is connected, then $d_{C/D}$ is the same as the different
exponent appearing in Lemma \ref{lem:Delta transform} (note that
since $C$ and $D$ have the same algebraically closed residue field,
the ramification index of a cover $C\to D$ is equal to its degree.)
\begin{lem}
We have
\[
\bv_{V}(E)=\frac{d_{C/D}}{2}=\frac{1}{2}\sum_{j=1}^{l}d_{C_{j}/D}.
\]
\end{lem}
\begin{proof}
This follows from \cite[Lemma 3.4]{MR2354797} and \cite[Theorem 4.7]{Wood-Yasuda-I}.
\end{proof}
We have an isomorphism 
\begin{align*}
\alpha:\Xi_{F}=\Hom_{\cO_{D}}^{H}(M,\cO_{F}) & \to\bigoplus_{j=1}^{l}\cO_{C_{j}}\\
\phi & \mapsto(\phi(x_{1}),\dots,\phi(x_{l})).
\end{align*}
For each $e\ge0$, we choose an element $\rho_{j,e}\in\cO_{C_{j}}$
with $v_{C_{j}}(\rho_{j,e})=e$, where $v_{C_{j}}$ is the normalized
valuation of $K(C_{j})$. The elements 
\[
\rho_{j,e}\quad(0\le e<\sharp O_{j}=[C_{j}:D])
\]
form a basis of $\cO_{C_{j}}$ as an $\cO_{D}$-module, and 
\[
\sigma_{j,e}:=(0,\dots,0,\overset{\underset{\smallsmile}{j}}{\rho_{j,e}},0,\dots,0)\quad(1\le j\le l,\,0\le e<\sharp O_{j})
\]
form a basis of $\bigoplus_{j=1}^{l}\cO_{C_{j}}$. Let $\psi_{j,e}\in M^{\left\langle F\right\rangle }$,
$1\le j\le l$, $0\le e<\sharp O_{j}$, be the dual basis of $\sigma_{j,e}$
through the isomorphism $\alpha$. The map $u_{F}^{*}:M\to M^{\left\langle F\right\rangle }$
sends $x_{i}$ with $i\in O_{j}$ to 
\[
\sum_{e=0}^{\sharp O_{j}-1}(\rho_{j,e}\cdot h_{i})\psi_{j,e},
\]
where $h_{i}$ is any element of $H$ sending $x_{j}$ to $x_{i}$
as above, and $f$ to
\begin{align*}
u_{F}^{*}(f) & =\sum_{j=1}^{l}\sum_{e=0}^{\sharp O_{j}-1}\left(\sum_{i\in O_{j}}f_{i}(\rho_{j,e}h_{i})\right)\psi_{j,e}\\
 & =\sum_{\substack{1\le j\le l\\
0\le e<\sharp O_{j}
}
}\Tr_{C_{j}/D}(f_{j}\rho_{j,e})\psi_{j,e}.
\end{align*}
Here $\Tr_{C_{j}/D}$ is the trace map $K(C_{j})\to K(D)$. 
\begin{lem}
Let $B\to D$ be a connected cover of degree $n$. For $e\in\ZZ_{\ge0}$,
we have 
\[
\Tr_{B/D}(\fm_{B}^{e})=\fm_{D}^{\left\lfloor \frac{e+d_{B/D}}{n}\right\rfloor }.
\]
Here $\left\lfloor r\right\rfloor $ is the largest integer $\le r$.
In particular, there exists a generator $\rho_{e}$ of $\fm_{B}^{e}$
such that 
\[
v_{D}(\Tr_{B/D}(\rho_{e}))=\left\lfloor \frac{e+d_{B/D}}{n}\right\rfloor 
\]
with $v_{D}$ the normalized valuation of $K(D)$. \end{lem}
\begin{proof}
From \cite[Proposition 7, page 50]{MR554237}, for $a\in\ZZ$, 
\begin{align*}
\Tr_{B/D}(\fm_{B}^{e})\subset\fm_{D}^{a} & \Leftrightarrow\fm_{B}^{e}\subset\fm_{B}^{an-d_{B/D}}\\
 & \Leftrightarrow a\le\frac{e+d_{B/D}}{n}.
\end{align*}
This shows the first assertion. To show the second assertion, suppose
by contrary that there does not exist such a generator of $\fm_{B}^{e}$.
From the first assertion, there exists an element $\tau\in\fm_{B}^{e+1}$
with 
\[
v_{D}(\Tr_{B/D}(\tau))=\left\lfloor \frac{e+d_{B/D}}{n}\right\rfloor .
\]
For any generator $\rho$ of $\fm_{B}^{e}$, $\rho+\tau$ is a generator
with the desired property, a contradiction. \end{proof}
\begin{prop}
\label{prop: compute b hyperplane permutation}Let us write $u_{F}^{*}(f)=\pi_{F}^{b}\phi$
with $\phi$ irreducible (a linear form over $\cO_{D}$ with at least
one coefficient a unit). Namely $b$ is the order of $u_{F}^{*}(f)$
along $V_{0}^{\left\langle F\right\rangle }$. Then 
\begin{align*}
b & =\sharp H\cdot\min\left\{ v_{D}(f_{j})+\left\lfloor \frac{d_{C_{j}/D}}{[C_{j}:D]}\right\rfloor \mid1\le j\le l\right\} .
\end{align*}
Here we put $v_{D}(0):=+\infty$ by convention. \end{prop}
\begin{proof}
From the lemma above, for a suitable choice of $\rho_{j,e}$, we have
\[
v_{D}\left(\Tr_{C_{j}/D}(f_{j}\rho_{j,e})\right)=\left\lfloor \frac{v_{C_{j}}(f_{j})+e+d_{C_{j}/D}}{[C_{j}:D]}\right\rfloor .
\]
Then 
\begin{align*}
b & =\min\{v_{F}(\Tr_{C_{j}/D}(f_{j}\rho_{j,e}))\mid1\le j\le l,\,0\le e<\sharp O_{j}\}\\
 & =\sharp H\cdot\min\left\{ \left\lfloor \frac{v_{C_{j}}(f_{j})+e+d_{C_{j}/D}}{[C_{j}:D]}\right\rfloor \mid1\le j\le l,\,0\le e<\sharp O_{j}\right\} \\
 & =\sharp H\cdot\min\left\{ v_{D}(f_{j})+\left\lfloor \frac{d_{C_{j}/D}}{[C_{j}:D]}\right\rfloor \mid1\le j\le l\right\} 
\end{align*}
which shows the proposition.\end{proof}
\begin{cor}
\label{cor:weight non-permutation}For $E\in\GCov(D)$, we have
\begin{align*}
\bv_{\sv}(E) & =\frac{1}{2}\sum_{j=1}^{l}d_{Cj/D}-\min\left\{ v_{D}(f_{j})+\left\lfloor \frac{d_{C_{j}/D}}{[C_{j}:D]}\right\rfloor \mid1\le j\le l\right\} .
\end{align*}
In particular, if $f=x_{1}+x_{2}+\cdots+x_{d}$, then 
\begin{align*}
\bv_{\sv}(E) & =\frac{1}{2}\sum_{j=1}^{l}d_{C_{j}/D}-\min\left\{ \left\lfloor \frac{d_{C_{j}/D}}{[C_{j}:D]}\right\rfloor \mid1\le j\le l\right\} .
\end{align*}
\end{cor}
\begin{proof}
In our situation, the symbol $\sv^{|F|}$ has, a priori, two meanings:
one is obtained by applying the untwisting technique directly to $\sv$
and the other by first applying it to $V$ and taking the induced
subvariety in $V^{|F|}$. However the two constructions actually coincide.
Indeed, if $\mathsf{m}$ is the linear part of $\cO_{\sv}$, then
we have a surjection $M\twoheadrightarrow\mathsf{m}$. It induces
a surjection $M^{|F|}\twoheadrightarrow\mathsf{m}^{|F|}$ and a closed
immersion $\AA_{D}^{d-1}\hookrightarrow\AA_{D}^{d}$. This shows the
claim. Therefore there is no confusion in the use of the symbol as
well as $\fv^{|F|}$. 

The boundary of $\fv^{|F|}$ is $-\bv_{\sv}(E)\cdot\sv_{0}^{|F|}$
from Lemma \ref{lem:Delta transform}, while 
\[
\left(\min\left\{ v_{D}(f_{j})+\left\lfloor \frac{d_{C_{j}/D}}{[C_{j}:D]}\right\rfloor \mid1\le j\le l\right\} -\bv_{V}(E)\right)\cdot\sv_{0}^{|F|}
\]
from Propositions \ref{prop: complete intersection} and \ref{prop: compute b hyperplane permutation}.
Comparing the coefficients shows the corollary.\end{proof}
\begin{rem}

\begin{enumerate}
\item Let $p$ denote the characteristic of $k$. If $p\nmid[C_{j}:D]$,
then $d_{C_{j}/D}=[C_{j}:D]-1$ and $\left\lfloor \frac{d_{C_{j}/D}}{[C_{j}:D]}\right\rfloor =0$.
Therefore, if $p\nmid d$ and if $f=x_{1}+\cdots+x_{d}$, then since
at least one $C_{j}$ satisfies $p\nmid[C_{j}:D]$, we have $\bv_{\sv}=\bv_{V}$.
This equality is also explained as follows. We have the exact sequence
\[
0\to\sv\to V\xrightarrow{(x_{1},\dots,x_{d})\mapsto\sum x_{i}}\AA_{D}^{1}\to0,
\]
whether we have $p\mid d$. If $p\nmid d$, this sequence splits.
The equality follows from the additivity of $\bv_{\bullet}$ (see
\cite{Wood-Yasuda-I}).
\item If, for some $j$, $f_{j}$ is a unit and $\sharp O_{j}=1$ (hence
$C_{j}=D$ and $d_{C_{j}/D}=0$), then the corollary above deduces
that $\bv_{\sv}=\bv_{V}.$ Again $V$ is isomorphic to the direct
sum of $\sv$ and a one-dimensional trivial representation, this time,
as an $H$-representation.
\end{enumerate}
\end{rem}
\begin{example}
Let $p$ be a prime number and $G=\left\langle g\right\rangle \cong\ZZ/p\ZZ$.
Suppose that $\cO_{D}=k[[\pi]]$ with $k$ of characteristic $p$
and that $G$ acts on $V=\Spec k[[\pi]][x_{1},\dots,x_{p}]$ by 
\[
g(x_{i})=\begin{cases}
x_{i+1} & (1\le i<p)\\
x_{1} & (i=p).
\end{cases}
\]
and that $\sv\subset V$ is the hyperplane defined by $f=x_{1}+\cdots+x_{p}$.
Let $E\in\GCov(D)$ be a connected $G$-cover. The \emph{ramification
jump} $j\in\ZZ_{>0}$ of $E$ is given by 
\[
j:=v_{D}(\pi_{E}g-\pi_{E})-1,
\]
which is not divisible by $p$. From \cite[page 83, Lemma 3]{MR554237},
\[
d_{E/D}=(p-1)(j+1).
\]
Accordingly,
\begin{align*}
\bv_{\sv}(E) & =\frac{d_{E/D}}{2}-\left\lfloor \frac{d_{E/D}}{p}\right\rfloor \\
 & =\frac{(p-1)(j+1)}{2}-\left\lfloor \frac{(p-1)(j+1)}{p}\right\rfloor \\
 & =\left(\frac{(p-1)(j-1)}{2}+(p-1)\right)-\left(1+\left\lfloor \frac{(p-1)j}{p}\right\rfloor \right)\\
 & =(p-2)+\left(\frac{(p-1)(j-1)}{2}-\left\lfloor \frac{(p-1)j}{p}\right\rfloor \right)\\
 & =(p-2)+\sum_{i=1}^{p-2}\left\lfloor \frac{ij}{p}\right\rfloor .
\end{align*}
The last equality follows from
\[
\frac{(p-1)(j-1)}{2}=\sum_{i=1}^{p-1}\left\lfloor \frac{ij}{p}\right\rfloor 
\]
(for instance, see \cite[page 94]{MR1001562}). Since $\codim(\sv_{0}^{G},\sv)=p-2$,
\[
\bw_{\sv}(E)=(p-2)-\bv_{\sv}(E)=-\sum_{i=1}^{p-2}\left\lfloor \frac{ij}{p}\right\rfloor ,
\]
which coincides with computation in \cite{Yasuda:2012fk} (see also
\cite{Yasuda:2013fk}).
\end{example}

\section{Some $S_{4}$-masses in characteristic two\label{sec: example S4}}

In this section, we consider the case where $\cO_{D}$ has characteristic
two, $G$ is the symmetric group $S_{4}$, $V:=\Spec\cO_{D}[x_{1},x_{2},x_{3},x_{4}]$
with the standard $G$-action and $\sv\subset V$ is the hyperplane
defined by $f=x_{1}+x_{2}+x_{3}+x_{4}$. The induced $G$-action on
$\sv$ is still faithful, since $\sv$ contains a point whose coordinates
are distinct one another, for instance, $(0,1,a,a+1)$ with $a\in k\setminus\{0,1\}$.
As an application of the computation of $\bv_{\sv}$ in the last section,
we will compute motivic integrals
\[
\MM=\int_{\GCov(D)}\LL^{-3\bv_{\sv}}\, d\tau\text{ and }\MM':=\int_{\GCov(D)}\LL^{3\bw_{\sv}}\, d\tau
\]
under some assumptions and observe that $\MM$ and $\MM'$ are dual
to each other. Such a duality was first observed in \cite{Wood-Yasuda-I}
and will be discussed in more details in \cite{Wood-Yasuda-II}. It
is also related to the Poincar\'{e} duality of stringy motifs. The
number 3 in the integrals was chosen, because for $n=1,2$, integrals
$\int_{\GCov(D)}\LL^{-n\bv_{\sv}}\, d\tau$ and $\int_{\GCov(D)}\LL^{n\bw_{\sv}}\, d\tau$
diverge. 

To compute $\MM$ and $\MM'$, we decompose them into the sums of
5 terms respectively. For $n\ge0$, let $\Fie_{n}$ and $\Eta_{n}$
be the (conjectural) moduli spaces of degree $n$ field extensions
and \'{e}tale extensions of $K(D)$ respectively. Since $G=S_{4}$,
giving a continuous homomorphism $\Gal(K(D)^{\sep}/K(D))\to G$ is
equivalent to giving a continuous $\Gal(K(D)^{\sep}/K(D))$-action
on $\{1,..,n\}$. Therefore the map
\begin{eqnarray*}
\GCov(D) & \to & \Eta_{4}\\
E & \mapsto & \cO_{D}\otimes_{\cO_{D}}K(D)
\end{eqnarray*}
is bijective. Since there are exactly 5 partitions of 4,
\[
(4),\,(3,1),\,(2^{2}),\,(2,1^{2}),\,(1^{4}).
\]
and $\Fie_{1}$ is a singleton, we have the following decomposition
of $\Eta_{4}$,
\begin{align*}
\Eta_{4} & \cong\Fie_{4}\sqcup\left(\Fie_{3}\times\Fie_{1}\right)\sqcup\frac{(\Fie_{2})^{2}}{\iota}\sqcup\left(\Fie_{2}\times\left(\Fie_{1}\right)^{2}\right)\sqcup\left(\Fie_{1}\right)^{4}\\
 & \cong\Fie_{4}\sqcup\Fie_{3}\sqcup\frac{(\Fie_{2})^{2}}{\iota}\sqcup\Fie_{2}\sqcup\{1\pt\}.
\end{align*}
Here $\iota$ is the involution of $(\Fie_{2})^{2}$ given by the
transposition of components. We have the corresponding stratification
\[
\GCov(D)=\bigsqcup_{\bp}\GCov(D)_{\bp},
\]
where $\bp$ runs over the partitions of $4$ and the corresponding
decompositions of $\MM$ and $\MM'$, 
\begin{gather*}
\MM=\sum_{\bp}\MM_{\bp}\text{ and }\MM'=\sum_{\bp}\MM'_{\bp}.
\end{gather*}
 To further computations, we need to assume the following conjecture.
\begin{conjecture}[The motivic version of Krasner's formula]
\label{conj: motivic Krasner}Suppose that $k$ has characteristic
$p>0$. Let $m\ge2$ be an integer and $\Fie_{m,d}\subset\Fie_{m}$
the locus of degree $m$ field extensions of $k((\pi))$ with discriminant
exponent $d$. Then we have the equality in $\cR$,
\[
[\Fie_{m,d}]=\begin{cases}
1 & (p\nmid m,\, d=m-1)\\
0 & (p\nmid m,\, d\ne m-1)\\
(\LL-1)\LL^{\left\lfloor (d-m+1)/p\right\rfloor } & (p\mid m,\, p\nmid(d-m+1))\\
0 & (p\mid m,\, p\mid(d-m+1)).
\end{cases}
\]
 
\end{conjecture}
Krasner \cite{MR0225756} showed that if $q=p^{e}$ is a power of
a prime number $p$, then the number of totally ramified degree $m$
extensions of the power series field $\FF_{q}((\pi))$ in its algebraic
closure $\overline{\FF_{q}((\pi))}$ is exactly
\[
\begin{cases}
m & (p\nmid m,\, d=m-1)\\
0 & (p\nmid m,\, d\ne m-1)\\
m(q-1)q^{\left\lfloor (d-m+1)/p\right\rfloor } & (p\mid m,\, p\nmid(d-m+1))\\
0 & (p\mid m,\, p\mid(d-m+1)).
\end{cases}
\]
Counting isomorphism classes (with weights coming from automorphisms)
rather than subfields of $\overline{\FF_{q}((\pi))}$ as did in \cite{MR500361},
we can kill the factor $m$. The conjecture above seems to be the
only reasonable possibility. 

In what follows, we exhibit how to compute $\MM_{(2^{2})}$ and $\MM_{(2^{2})}'$.
Computations of the other terms are similar and easier. We go back
to the assumption that $k$ has characteristic two. If $d=2n+m$,
then the conjecture reads 
\[
[\Fie_{m,d}]=(\LL-1)\LL^{n}.
\]
Let $E\in\GCov(D)_{2,2}$ and $C=C_{1}\sqcup C_{2}$ the associated
quartic cover of $D$, where $C_{1}$ and $C_{2}$ are double covers
of $D$ with $d_{C_{1}/D}\le d_{C_{2}/D}$. Then 
\[
\bv_{\sv}(E)=\frac{d_{C_{1}/D}+d_{C_{2}/D}}{2}-\left\lfloor \frac{d_{C_{1}/D}}{2}\right\rfloor .
\]
If we write $d_{C_{1}/D}=2n+2$ and $d_{C_{2}/D}=2m+2$ for some $m\ge n\ge0$,
\[
\bv_{\sv}(E)=m+1.
\]
From the last property of $\cR$ in the list in Section \ref{sub:Motivic-integration},
we have 
\[
\left[\frac{(\GG_{m})^{n}}{S_{n}}\right]=\left[\frac{\AA_{k}^{n}}{S_{n}}\right]-\left[\frac{\AA_{k}^{n-1}}{S_{n-1}}\right]=\LL^{n}-\LL^{n-1}.
\]
Accordingly, 
\[
\left[\frac{\left(\Fie_{2,2n+2}\right)^{2}}{\iota}\right]=(\LL-1)\LL^{2n+1}.
\]
We have
\begin{align*}
\MM_{(2^{2})} & =\sum_{m=0}^{\infty}\sum_{n=0}^{m-1}(\LL-1)^{2}\LL^{n+m}\cdot\LL^{-3(m+1)}+\sum_{m=0}^{\infty}(\LL-1)\LL^{2m+1}\cdot\LL^{-3(m+1)}\\
 & =(\LL-1)^{2}\LL^{-3}\sum_{m=0}^{\infty}\LL^{-2m}\sum_{n=0}^{m-1}\LL^{n}+(\LL-1)\LL^{-2}\sum_{m=0}^{\infty}\LL^{-m}\\
 & =(\LL-1)^{2}\LL^{-3}\sum_{m=0}^{\infty}\LL^{-2m}\cdot\frac{\LL^{m}-1}{\LL-1}+(\LL-1)\LL^{-2}\cdot\frac{\LL}{\LL-1}\\
 & =(\LL-1)\LL^{-3}\sum_{m=0}^{\infty}\left(\LL^{-m}-\LL^{-2m}\right)+\LL^{-1}\\
 & =\frac{\LL^{-2}+\LL^{-1}+1}{\LL+1}.
\end{align*}
If we suppose that the $H$-orbits in $\{x_{1},x_{2},x_{3},x_{4}\}$
are $\{x_{1},x_{3}\}$ and $\{x_{2},x_{4}\}$, then 
\[
\AA_{k}^{2}\cong\{(x,y,x,y)\mid x,y\in k\}=V_{0}^{H}\subset\sv_{0}^{H},
\]
we have $\codim(\sv_{0}^{H},\sv_{0})=1$ and
\[
\MM_{(2^{2})}'=\LL^{3}\cdot\MM_{(2^{2})}=\frac{\LL+\LL^{2}+\LL^{3}}{\LL+1}.
\]
Thus $\MM_{(2^{2})}$ and $\MM_{(2^{2})}'$ are dual to each other
in the sense that they interchange by substituting $\LL^{-1}$ for
$\LL$. 

For the other terms $\MM_{\bp}$ and $\MM_{\bp}'$, we see
\begin{align*}
\MM_{(4)} & =\LL^{-4}+\LL^{-2} & \MM_{(4)}' & =\LL^{4}+\LL^{2}\\
\MM_{(3,1)} & =\LL^{-3} & \MM_{(3,1)}' & =\LL^{3}\\
\MM_{(2,1^{2})} & =\frac{\LL^{-1}}{\LL+1} & \MM_{(2,1^{2})}' & =\frac{\LL^{2}}{\LL+1}\\
\MM_{(1^{4})} & =1 & \MM_{(1^{4})}' & =1.
\end{align*}
For each partition $\bp$, we would have the duality. Summing these
up, we get
\begin{gather*}
\MM=\LL^{-4}+\LL^{-3}+\LL^{-2}+1+\frac{\LL^{-2}+2\LL^{-1}+1}{\LL+1},\\
\MM'=\LL^{4}+\LL^{3}+\LL^{2}+1+\frac{\LL+2\LL^{2}+\LL^{3}}{\LL+1}.
\end{gather*}

\begin{rem}
By similar computations, we can easily deduce the motivic counterpart
of Serre's mass formula \cite{MR500361} from Conjecture \ref{conj: motivic Krasner}:
in any characteristic and for any $m$,
\[
\int_{\Fie_{m}}\LL^{-d}\, d\tau=\LL^{1-m}.
\]
Here $d:\Fie_{m}\to\ZZ$ is the function associating the discriminant
exponent to a field extension and $\tau$ is the tautological motivic
measures on $\Fie_{n}$. This too justifies the conjecture. With some
more computation, it would be possible to get also the motivic version
of Bhargava's formula \cite{MR2354798}.
\end{rem}

\section{Concluding remarks\label{sec:Concluding-remarks}}

We will end the paper by making some remarks and raising several problems
for the future.

\subsection{Singularities of $\sv$, $\sv^{|F|}$ and $\sv^{|F|,\nu}$}

In the definition of log varieties, we assumed that the ambient variety
is always normal. It forced us to take the normalization $\sv^{|F|,\nu}$
of the untwisting variety $\sv^{|F|}$. The normality assumption enables
us to work in a standard setting of the minimal model program and
to use familiar computations of divisors. However this restriction
seems not to be really necessary. For instance, we can define the
stringy motif if we specify an invertible subsheaf of 
\[
\left(\bigwedge^{d}\Omega_{X/D}\right)^{\otimes r}\otimes K(X)
\]
rather than a boundary divisor $\Delta$. We then would be able to
replace most of arguments in this paper with ones using subsheaves
rather than divisors.

What kind of singularities can $\sv^{|F|}$ and $\sv^{|F|,\nu}$ have?
In the examples in Sections \ref{sec:A-tame-example} and \ref{sec:A-wild-example},
rather mild singularities appeared. Indeed, in both examples, for
every $E\in\GCov(D)$, the untwisting variety $\sv^{|F|}$ had only
normal hypersurface singularities having a crepant resolution. In
general, if $\sv\subset V$ is a hypersurface, then so is $\sv^{|F|}\subset V^{|F|}$,
although the author does not know if it is always normal. What about
complete intersections? If the answer is positive, then we would be
able to use Proposition \ref{prop: complete intersection} to compute
the boundary of $\fv^{|F|}$. Moreover, we might be able to generalize,
for instance, the semi-continuity of the minimal log discrepancies
to quotients of local complete intersections by combining arguments
used for local complete intersections \cite{MR2000468,MR2102399}
and quotient singularities \cite{Nakamura:2013fk}. 

In the tame case, if $\cO_{D}=k[[\pi]]$, then, as we saw in Section
\ref{sec:A-tame-example}, the map $u^{*}:\cO_{V}\to\cO_{V^{\left\langle F\right\rangle }}$
is simply given by $x_{i}\mapsto\pi^{a_{i}}\rx_{i}$, $a_{i}\in\QQ$
for a suitable choice of coordinates $x_{1},\dots,x_{d}\in\cO_{V}$
and $\rx_{1},\dots,\rx_{d}\in\cO_{V^{\left\langle F\right\rangle }}$.
Therefore, if $\sv\subset V$ is defined by $f_{1},\dots,f_{l}\in\cO_{V}$,
then the scheme-theoretic preimage $u^{-1}(\sv)\subset V^{\left\langle F\right\rangle }$
is defined by $u^{*}f_{1},\dots,u^{*}f_{l}$, which have the same
number of terms with $f_{1},\dots,f_{l}$ respectively. In particular,
if $\sv$ is an affine toric variety, then it is embedded into $V$
as a closed subvariety defined by binomials $f_{1},\dots,f_{l}$,
and then $u^{-1}(\sv)$ is also defined by binomials. Thanks to this
fact, we might be able to study $\sv^{|F|}$ from the combinatorial
viewpoint.

In the example in Section \ref{sec:A-wild-example}, $\sv^{|F|}$
had $A_{1}$-singularities and $D_{2n}^{0}$-singularities, from Artin's
classification of rational double points in positive characteristic
\cite{MR0450263}. In general, when $\sv$ and hence $\sv^{|F|}$
are surfaces (relative dimension one over $D$), then what kind of
singularities can $\sv^{|F|}$ have? Does every rational double point
appear on some $\sv^{|F|}$? If we can compute singularities of $\sv^{|F|}$
systematically, we would be able to compute the right side of the
equality in Conjecture \ref{conj: non-linear McKay-1} explicitly
and to derive many mass formulas, explained below.

\subsection{Mass formulas for extensions of a local field and local Galois representations}

For a constructible function $\Phi:\GCov(D)\to\cR$, the integral
\[
\int_{\GCov(D)}\Phi\, d\tau
\]
can be regarded as the motivic count of $G$-covers of $D$ with $E\in\GCov(D)$
weighted by $\Phi(E)$. If $\cO_{D}$ has a finite residue field $k=\FF_{q}$
rather than algebraically closed one, then the motif $\int_{\GCov(D)}\Phi\, d\tau$
should give an actual weighted count of $G$-covers of $D$ as its
\emph{point-counting realization.} This observation was made in \cite{Yasuda:2012fk,Wood-Yasuda-I}
in the context of the wild McKay correspondence for linear actions.
Such counts are number-theoretic problems by nature. Indeed, as clarified
in \cite{Wood-Yasuda-I}, counts appearing in the McKay correspondence
are closely related to counts of extensions of a local filed and to
counts of local Galois representations, for instance, studied in \cite{MR0225756,MR500361,MR2354798,MR2354797,MR2411405}:
formulas for such counts are called \emph{mass formulas.} The weights
previously considered have the form $\LL^{\alpha}$ for some function
$\alpha:\GCov(D)\to\QQ$, corresponding to weights of the form $\frac{1}{\sharp H}q^{\alpha}$
in actual counts if $k=\FF_{q}$. However, in Conjecture \ref{conj: non-linear McKay-1},
we have fancier weights $M_{\st,C_{G}(H)}(\fv^{|F|,\nu})$, which
are expected to be often rational functions in $\LL$ or $\LL^{1/n}$,
$n\in\ZZ_{>0}$ (it is actually the case in examples in Sections \ref{sec:A-tame-example}
and \ref{sec:A-wild-example}). The new weights clearly have geometric
meaning and might provide some insight to the number theory.

\subsection{Weight functions for general representations}

How to compute functions $\bw_{V}$ and $\bv_{V}$ for general linear
actions $G\curvearrowright V=\AA_{D}^{d}$? By now, we have satisfactory
answers in the following cases:
\begin{itemize}
\item the tame case \cite{Yasuda:2013fk,Wood-Yasuda-I},
\item the case where $\cO_{D}=k[[\pi]]$ with $k$ of characteristic $p>0$
and $G=\ZZ/p\ZZ$ \cite{Yasuda:2012fk}, 
\item permutation representations \cite{Wood-Yasuda-I},
\item a hyperplane in a permutation representation defined by an invariant
linear form (Corollary \ref{cor:weight non-permutation}).
\end{itemize}
For the general case, we can always embed a given representation into
a permutation representation and apply Corollary \ref{prop: complete intersection}.
The problem is to compute the divisor $A_{F}$, which appears to be
almost equivalent to computing functions $\bw_{V}$ and $\bv_{V}$.

\subsection{The convergence or divergence of motivic integrals}

Motivic integrals discussed in this paper do not generally converge
and stringy motifs and motivic masses can be infinite. In such a case,
the wild McKay correspondence conjecture does not mean much (but not
none). In characteristic zero, the convergence of a stringy motif
is equivalent to that the given singularities are (Kawamata) log terminal.
The author then called \emph{stringily log terminal} singularities
whose stringy motif converges, which is equivalent to the usual notion
of log terminal if singularities admit a log resolution. Since quotient
singularities in characteristic zero are always log terminal, this
divergence problem did not occur in the study of the McKay correspondence
in characteristic zero. However wild quotient singularities are sometimes
stringily log terminal and sometimes not. It is an interesting problem
to know when they are and when they are not. If $\cO_{D}=k[[\pi]]$
has characteristic $p>0$ and $G=\ZZ/p\ZZ$, then the convergence
is determined by the value of a simple representation-theoretic invariant
denoted by $D_{V}$ in \cite{Yasuda:2012fk}. Is it possible to generalize
this invariant to other groups?

Another problem is to attach finite values to divergent motivic integrals
by ``renormalizing'' them somehow, for instance, as tried by Veys
\cite{MR2008717,MR2030094} for stringy invariants in characteristic
zero. 

\bibliographystyle{amsalpha}
\bibliography{/Users/highernash/Dropbox/Math_Articles/mybib}

\providecommand{\bysame}{\leavevmode\hbox to3em{\hrulefill}\thinspace}
\providecommand{\MR}{\relax\ifhmode\unskip\space\fi MR }
\providecommand{\MRhref}[2]{%
  \href{http://www.ams.org/mathscinet-getitem?mr=#1}{#2}
}
\providecommand{\href}[2]{#2}
\begin{thebibliography}{EMY03}

\bibitem[Art77]{MR0450263}
M.~Artin, \emph{Coverings of the rational double points in characteristic
  {$p$}}, Complex analysis and algebraic geometry, Iwanami Shoten, Tokyo, 1977,
  pp.~11--22. \MR{MR0450263 (56 \#8559)}

\bibitem[Bat98]{MR1672108}
Victor~V. Batyrev, \emph{Stringy {H}odge numbers of varieties with {G}orenstein
  canonical singularities}, Integrable systems and algebraic geometry
  (Kobe/Kyoto, 1997), World Sci. Publ., River Edge, NJ, 1998, pp.~1--32.
  \MR{MR1672108 (2001a:14039)}

\bibitem[Bat99]{MR1677693}
\bysame, \emph{Non-{A}rchimedean integrals and stringy {E}uler numbers of
  log-terminal pairs}, J. Eur. Math. Soc. (JEMS) \textbf{1} (1999), no.~1,
  5--33. \MR{MR1677693 (2001j:14018)}

\bibitem[BD96]{MR1404917}
Victor~V. Batyrev and Dimitrios~I. Dais, \emph{Strong {M}c{K}ay correspondence,
  string-theoretic {H}odge numbers and mirror symmetry}, Topology \textbf{35}
  (1996), no.~4, 901--929. \MR{1404917 (97e:14023)}

\bibitem[Bha07]{MR2354798}
Manjul Bhargava, \emph{Mass formulae for extensions of local fields, and
  conjectures on the density of number field discriminants}, Int. Math. Res.
  Not. IMRN (2007), no.~17, Art. ID rnm052, 20. \MR{2354798 (2009e:11220)}

\bibitem[DL99]{MR1664700}
Jan Denef and Fran{\c{c}}ois Loeser, \emph{Germs of arcs on singular algebraic
  varieties and motivic integration}, Invent. Math. \textbf{135} (1999), no.~1,
  201--232. \MR{MR1664700 (99k:14002)}

\bibitem[DL02]{MR1905024}
\bysame, \emph{Motivic integration, quotient singularities and the {M}c{K}ay
  correspondence}, Compositio Math. \textbf{131} (2002), no.~3, 267--290.
  \MR{MR1905024 (2004e:14010)}

\bibitem[Eis95]{MR1322960}
David Eisenbud, \emph{Commutative algebra: with a view toward algebraic
  geometry}, Graduate Texts in Mathematics, vol. 150, Springer-Verlag, New
  York, 1995. \MR{MR1322960 (97a:13001)}

\bibitem[EM04]{MR2102399}
Lawrence Ein and Mircea Musta{\c{t}}{\v{a}}, \emph{Inversion of adjunction for
  local complete intersection varieties}, Amer. J. Math. \textbf{126} (2004),
  no.~6, 1355--1365. \MR{2102399 (2005j:14020)}

\bibitem[EMY03]{MR2000468}
Lawrence Ein, Mircea Musta{\c{t}}{\u{a}}, and Takehiko Yasuda, \emph{Jet
  schemes, log discrepancies and inversion of adjunction}, Invent. Math.
  \textbf{153} (2003), no.~3, 519--535. \MR{MR2000468 (2004f:14028)}

\bibitem[GKP89]{MR1001562}
Ronald~L. Graham, Donald~E. Knuth, and Oren Patashnik, \emph{Concrete
  mathematics}, Addison-Wesley Publishing Company Advanced Book Program,
  Reading, MA, 1989, A foundation for computer science. \MR{1001562
  (91f:00001)}

\bibitem[Ked07]{MR2354797}
Kiran~S. Kedlaya, \emph{Mass formulas for local {G}alois representations}, Int.
  Math. Res. Not. IMRN (2007), no.~17, Art. ID rnm021, 26, With an appendix by
  Daniel Gulotta. \MR{2354797 (2008m:11231)}

\bibitem[Kol13]{MR3057950}
J{{\'a}}nos Koll{{\'a}}r, \emph{Singularities of the minimal model program},
  Cambridge Tracts in Mathematics, vol. 200, Cambridge University Press,
  Cambridge, 2013, With a collaboration of S{{\'a}}ndor Kov{{\'a}}cs.
  \MR{3057950}

\bibitem[Kon95]{Kontsevich-motivic}
Maxim Kontsevich, \emph{Lecture at {O}rsay}.

\bibitem[Kra66]{MR0225756}
Marc Krasner, \emph{Nombre des extensions d'un degr{\'e} donn{\'e} d'un corps
  {${\mathfrak {P}}$}-adique}, Les {T}endances {G}{\'e}om. en {A}lg{\`e}bre et
  {T}h{\'e}orie des {N}ombres, Editions du Centre National de la Recherche
  Scientifique, Paris, 1966, pp.~143--169. \MR{0225756 (37 \#1349)}

\bibitem[LP04]{MR2069013}
Ernesto Lupercio and Mainak Poddar, \emph{The global {M}c{K}ay-{R}uan
  correspondence via motivic integration}, Bull. London Math. Soc. \textbf{36}
  (2004), no.~4, 509--515. \MR{2069013 (2005c:14026)}

\bibitem[Nak]{Nakamura:2013fk}
Yusuke Nakamura, \emph{On semi-continuity problems for minimal log
  discrepancies}, arXiv:1305.1410.

\bibitem[Nic11]{MR2770561}
Johannes Nicaise, \emph{A trace formula for varieties over a discretely valued
  field}, J. Reine Angew. Math. \textbf{650} (2011), 193--238. \MR{2770561
  (2012d:14039)}

\bibitem[NS11]{MR2885336}
Johannes Nicaise and Julien Sebag, \emph{The {G}rothendieck ring of varieties},
  Motivic integration and its interactions with model theory and
  non-{A}rchimedean geometry. {V}olume {I}, London Math. Soc. Lecture Note
  Ser., vol. 383, Cambridge Univ. Press, Cambridge, 2011, pp.~145--188.
  \MR{2885336}

\bibitem[Seb04]{MR2075915}
Julien Sebag, \emph{Int{\'e}gration motivique sur les sch{\'e}mas formels},
  Bull. Soc. Math. France \textbf{132} (2004), no.~1, 1--54. \MR{2075915
  (2005e:14017)}

\bibitem[Ser78]{MR500361}
Jean-Pierre Serre, \emph{Une ``formule de masse'' pour les extensions
  totalement ramifi{\'e}es de degr{\'e} donn{\'e} d'un corps local}, C. R.
  Acad. Sci. Paris S{\'e}r. A-B \textbf{286} (1978), no.~22, A1031--A1036.
  \MR{500361 (80a:12018)}

\bibitem[Ser79]{MR554237}
\bysame, \emph{Local fields}, Graduate Texts in Mathematics, vol.~67,
  Springer-Verlag, New York, 1979, Translated from the French by Marvin Jay
  Greenberg. \MR{554237 (82e:12016)}

\bibitem[Vey03]{MR2030094}
Willem Veys, \emph{Stringy zeta functions for {$\Bbb Q$}-{G}orenstein
  varieties}, Duke Math. J. \textbf{120} (2003), no.~3, 469--514. \MR{2030094
  (2004m:14075)}

\bibitem[Vey04]{MR2008717}
\bysame, \emph{Stringy invariants of normal surfaces}, J. Algebraic Geom.
  \textbf{13} (2004), no.~1, 115--141. \MR{2008717 (2004h:14024)}

\bibitem[Woo08]{MR2411405}
Melanie~Matchett Wood, \emph{Mass formulas for local {G}alois representations
  to wreath products and cross products}, Algebra Number Theory \textbf{2}
  (2008), no.~4, 391--405. \MR{2411405 (2009e:11219)}

\bibitem[WYa]{Wood-Yasuda-I}
Melanie~Matchett Wood and Takehiko Yasuda, \emph{Mass formulas for local
  {G}alois representations and quotient singularities {I}: A comparison of
  counting functions}, arXiv:1309:2879.

\bibitem[WYb]{Wood-Yasuda-II}
\bysame, \emph{Mass formulas for local {G}alois representations and quotient
  singularities {II}: Dualities}, in preparation.

\bibitem[Yasa]{Yasuda:2012fk}
Takehiko Yasuda, \emph{The $p$-cyclic {M}c{K}ay correspondence via motivic
  integration}, arXiv:1208.0132, to appear in {\it Compositio Mathematica}.

\bibitem[Yasb]{Yasuda:2013fk}
\bysame, \emph{Toward motivic integration over wild {D}eligne-{M}umford
  stacks}, arXiv:1302.2982, to appear in the proceedings of ``Higher
  Dimensional Algebraic Geometry - in honour of Professor Yujiro Kawamata's
  sixtieth birthday".

\bibitem[Yas04]{MR2027195}
Takehiko Yasuda, \emph{Twisted jets, motivic measures and orbifold cohomology},
  Compos. Math. \textbf{140} (2004), no.~2, 396--422. \MR{MR2027195
  (2004m:14037)}

\bibitem[Yas06]{MR2271984}
\bysame, \emph{Motivic integration over {D}eligne-{M}umford stacks}, Adv. Math.
  \textbf{207} (2006), no.~2, 707--761. \MR{MR2271984 (2007j:14002)}

\end{thebibliography}

\end{document}